\documentclass[11pt]{article}

\usepackage{mathpazo}
\usepackage{amsfonts}
\usepackage{amsmath}
\usepackage{graphicx}
\usepackage{latexsym}
\usepackage{mathabx}
\usepackage[margin=1.05in]{geometry}
  
\usepackage[T1]{fontenc}
\usepackage{times}
\usepackage{color,graphicx}
\usepackage{array}
\usepackage{enumerate}
\usepackage{amsmath}
\usepackage{amssymb}
\usepackage{amsthm}
\usepackage{pgfplots}
\usepackage{pgf}
\usepackage{tikz}
\usetikzlibrary{patterns}
\usepgfplotslibrary{patchplots} % LATEX and plain TEX
\usetikzlibrary{pgfplots.patchplots} % LATEX and plain TEX
\pgfplotsset{width=9cm,compat=1.5.1}

\usepackage{enumitem}
\usepackage[english]{babel}
\usepackage{amsthm}
\numberwithin{equation}{section}
\newtheorem{lemma}{Lemma}[section]
\newtheorem{prop}[lemma]{Proposition}

\newtheorem{theorem}[lemma]{Theorem}
\newtheorem{cor}[lemma]{Corollary}

\newtheorem{rem}[lemma]{Remark}
\newtheorem{ex}[lemma]{Example}

\newcommand{\be}{\mathbf e}

\newcommand{\bx}{\mathbf x}
\newcommand{\bu}{\mathbf u}
\newcommand{\bv}{\mathbf v}

\newcommand{\bw}{\mathbf w}

\newcommand{\one}{\mathbf 1}
\newcommand{\zero}{\mathbf 0}
\newcommand{\supp}{{\Phi}}
\newcommand\spec{{\mathrm{spec}}}
\newcommand{\mspec}{\mathrm{mspec}}
\newcommand{\nspec}{\mathrm{nspec}}

\usepackage{xcolor}
\usepackage{makeidx}
\usepackage[pdftex,linktocpage=true,colorlinks,citecolor=blue,linkcolor=blue,pagebackref]{hyperref}
\usepackage{hyperref}
\usepackage{cleveref}
\usepackage[alphabetic,backrefs]{amsrefs}

\def\C{\mathbb{C}}
\def\R{\mathbb{R}}
\def\Q{\mathbb{Q}}
\def\Z{\mathbb{Z}}

\usepackage[hang,flushmargin]{footmisc}
\begin{document}

\title{Quantum walks and graph operations}

\author{Joy Cooper\textsuperscript{1}, Homer De Vera\textsuperscript{2}, Hermie Monterde\textsuperscript{3$\dagger$}, S.A.~Talebpour\textsuperscript{3}, Xuzhu (Ruth) Wang\textsuperscript{4,5,6}
}

\maketitle

\begin{abstract}
Let $U_X(t)$ be the transition matrix of a quantum walk on a graph $X$ relative to its adjacency matrix $A$ or the Laplacian matrix $L$. This paper investigates the behavior of quantum walks under Cartesian products, joins, and graph complements.
We have two main goals. First, we characterize the conditions such that peak state transfer and pretty good state transfer are preserved under these operations, allowing us to 
construct new families of graphs admitting these properties. Our second goal is to analyze the relationship between the quantum walks on a graph and its complement. We provide bounds for
\begin{center}
$f_{u,v}(t)=\big|U_{X^c}(t)_{u,v}-e^{it\delta}U_X(-t)_{u,v}\big|\quad $ and $\quad g_{u,v}(t)=\big||U_X(t)_{u,v}|-|U_{X^c}(t)_{u,v}|\big|$,
\end{center}
where $\delta=-1$ when dealing with $A$ and $\delta=n$ otherwise. Note that $f_{u,v}(t)$ and $g_{u,v}(t)$ both measure the difference between the behavior of quantum state transfer between vertices $u$ and $v$ in a graph and its complement. If $X$ is regular or $M=L$, then $f_{u,v}(t)$ is bounded above by $\frac{2}{|V(X)|}$. If $X$ is non-regular and $M=A$, then we utilize the main eigenvalues of a graph to obtain an upper bound for $f_{u,v}(t)$ which depends only on $A$. We also use the bounding matrix of the graph to give bounds for the Nordhaus-Gaddum type relations 
\begin{center}
$|U_X(t)_{u,v}|+|U_{X^c}(t)_{u,v}|\quad $ and $\quad |U_X(t)_{u,v}|\cdot |U_{X^c}(t)_{u,v}|$.
\end{center}
Finally, we demonstrate that most of our bounds %for $f_{u,v}(t)$, $g_{u,v}(t)$ $\sigma_{u,v}(t)$ and $\rho_{u,v}(t)$ 
are sharp for certain families of graphs.
\end{abstract}

\noindent \textbf{Keywords:} quantum walk, peak state transfer, pretty good state transfer, sedentary vertex, main eigenvalue\\%, split graph, join graph, Cartesian product, graph complement\\
	
\noindent \textbf{MSC2010 Classification:} %%normally can have max of 5
81P45; %Quantum information, communication, networks
05C50; %Graphs and linear algebra (matrices, eigenvalues, etc.)
15A16; %Matrix exponential
%05C22; %Signed and  graphs
15A18 %Eigenvalues, singular values, and eigenvectors

% Possible additiona, 05C81 Random walks on graphs, 

\addtocounter{footnote}{1}
\footnotetext{Department of Mathematics and Statistics, University of Victoria, Victoria, BC, Canada V8P 5C2}

\addtocounter{footnote}{1}
\footnotetext{Department of Mathematics, University of Manitoba, Winnipeg, MB, Canada R3T 2N2}

\addtocounter{footnote}{1}
\footnotetext{Department of Mathematics and Statistics, University of Regina, Regina, SK, Canada S4S 0A2}

\addtocounter{footnote}{1}
\footnotetext{Department of Mathematics, University of British Columbia, Vancouver, BC, Canada V6T 1Z2}

\addtocounter{footnote}{1}
\footnotetext{Department of Physics and Astronomy, University of Waterloo, Waterloo, ON, Canada N2L 3G1}

\addtocounter{footnote}{1}
\footnotetext{Perimeter Institute for Theoretical Physics, Waterloo, ON, Canada N2L 2Y5}

%{\def\thefootnote{}\footnotetext{Email addresses: {\tt joycooper@uvic.ca, deverahf@myumanitoba.ca, Hermie.Monterde@uregina.ca, stp664@uregina.ca, ruth.wang@uwaterloo.ca}}

{\def\thefootnote{$\dagger$}\footnotetext{Corresponding author {\tt (Hermie.Monterde@uregina.ca)}}}

\medskip

\section{Introduction}

Let $X$ be a graph with adjacency matrix $A=A(X)$ and Laplacian matrix $L=L(X)$. A \textit{(continuous) quantum walk} on $X$ with Hamiltonian $M\in\{A,L\}$ is determined by a one-parameter family of complex symmetric unitary matrices given by $U_X(t)=e^{itM(X)}$. The matrix $U_X(t)$ is called the \textit{transition matrix} of $X$, and describes the propagation of quantum states in the spin network represented by the graph $X$. Here, we view the vertices of the graph as the qubits in the network while the edges represent their interactions. Indeed, the unitarity of $U_X(t)$ allows us to interpret $|U_X(t)_{u,v}|^2$ as the \textit{transfer probability} between the vertices $u$ and $v$ in $X$ at time $t$. Thus, the higher the value of $|U_X(t)_{u,v}|$ for some time $t$, the higher the probability of finding a quantum state at vertex $v$ at time $t$ that was initially assigned at vertex $u$.

Quantum walks are important combinatorial objects with applications to search problems \cite{childs2004spatial}, algorithms \cite{farhi1998quantum}, computation \cite{childs2009universal} and state transfer \cite{Godsil2012a}. In this paper, we focus on state transfer, with the main goal of investigating the behavior of quantum walks under joins, Cartesian products, and graph complements. In particular, we concentrate on pretty good state transfer (PGST) and peak state transfer, since they extend results on the well-known topic of perfect state transfer (PST), a desirable type of state transfer that represents accurate transmission of quantum information within a spin network.

The paper is organized as follows. In Section~\ref{sec:qw}, we revisit and establish new basic results on three types of state transfer on graphs: peak state transfer, PGST, and vertex sedentariness. In Sections~\ref{sec:join},~\ref{sec:cartprod}, and~\ref{sec:comp}, we characterize the conditions such that PGST and peak state transfer in a graph are preserved in joins, Cartesian products, and graph complements, respectively. This allows us to produce new infinite families of graphs with PGST and peak state transfer. In Section~\ref{sec:compare}, we turn to analyzing the relationship between quantum walks on a graph and its complement. We provide tight upper and lower bounds for $f_{u,v}(t)=\big|U_{X^c}(t)_{u,v}-e^{it\delta}U_X(-t)_{u,v}\big|$ and $g_{u,v}(t)=\big||U_{X^c}(t)_{u,v}|-|U_X(t)_{u,v}|\big|$ where $\delta\in\R$ depends on $M$. These two expressions quantify the difference between the behavior of state transfer between two vertices in a graph and its complement. In particular, we show that if (i) $M=A$ and $X$ is regular, or (ii) $M=L$, then $f_{u,v}(t)$ and $g_{u,v}(t)$ are upper bounded by $\frac{2}{n}$, where $n$ is the number of vertices of $X$, which in turn implies that the behavior of the quantum walks on $X$ and its complement tend to coincide as we increase the number of vertices of $X$. In Section~\ref{sec:nonreg}, we derive a reasonable upper bound for $f_{u,v}(t)$ whenever $M=A$ and $X$ is irregular using the main eigenvectors of $X$. To the best of our knowledge, this is the first application of the main eigenvalues to the theory of quantum walks. We devote Section~\ref{sec:NG} to providing bounds for Nordhaus-Gaddum type expressions $\sigma_{u,v}(t)=|U_X(t)_{u,v}|+|U_{X^c}(t)_{u,v}|$ and $\rho_{u,v}(t)=|U_X(t)_{u,v}|\cdot |U_{X^c}(t)_{u,v}|$. In Section~\ref{sec:split}, we compute the spectral decomposition of bidegreed split graphs and identify new infinite families of graphs that admit sedentariness in the graph and its complement. These families are interesting because they are nonregular coconnected graphs that in general do not contain twins, and are in general neither joins nor Cartesian products, unlike most known examples of graphs with sedentary vertices in the literature. Finally, we present open questions in Section~\ref{sec:op}. The remainder of the present section is allotted to basic definitions and notation relevant to the rest of the paper.

Unless otherwise stated, we assume that our graphs are simple (loopless), connected, unweighted, and undirected. If $X$ is clear from the context, then we write $M(X)$, $A(X)$, and $L(X)$ respectively as $M$, $A$, and $L$. Since $M$ is a real symmetric matrix, we may write $M$ in its spectral decomposition as
\begin{equation}
\label{specdec}
M=\sum_{\lambda\in\spec(M)}\lambda E_{\lambda},
\end{equation}
where $\spec(M)$ denotes the set of distinct eigenvalues of $M$ and $E_{\lambda}$ is the spectral idempotent onto the eigenspace of $M$ associated with $\lambda\in \spec(M)$. Note that all such $E_\lambda$ are real, idempotent, pairwise multiplicatively orthogonal, and sum to identity. Using (\ref{specdec}), we may write the transition matrix as
\begin{equation*}
\label{specdecU}
U_X(t)=\sum_{\lambda\in\spec(M)}e^{it\lambda}E_{\lambda}.
\end{equation*}

The \textit{eigenvalue support} of a vertex $u$ in $X$, denoted $\supp_{u}$, is the set $\supp_{u}=\{\lambda\in\spec(M):E_{\lambda}\be_u\neq \textbf{0}\}.$ 

We let $\mathbf{1}$ and $\mathbf{J}$ (resp., $\zero$ and $O$) denote the all-ones (resp., zero) vector and matrix of appropriate size, respectively. We also denote the complete graph, cycle, and path on $n$ vertices by $K_n$, $C_n$ and $P_n$, respectively, and the complete bipartite graph with parts of sizes $r$ and $s$ by $K_{r,s}$. A graph $X$ is \textit{regular} if $\mathbf{1}$ is an eigenvector for $A$ (equivalently, all row sums of $A$ are equal). In particular, we say that $X$ is \textit{$k$-regular} if each row sum of $A$ is $k$. Hence, for a connected $k$-regular graph, $\frac{1}{n}\mathbf{J}$ is a spectral idempotent associated with the eigenvalue $k$. If $X$ is $k$-regular, then $L=kI-A$, and so if the transition matrices of $X$ relative to $A$ and $L$ are denoted $U(t)$ and $U'(t)$, respectively, then $|U(t)_{u,v}|=|U'(t)_{u,v}|$ for all $t$. So, the quantum walks relative to $A$ and $L$ are equivalent if $X$ is regular, i.e. they exhibit the same types of state transfer at the same time. We say that $X$ is \textit{$M$-integral} if $\spec(M)\subseteq\Z$. If $M$ is clear from the context, then we simply say integral. 

\section{Quantum state transfer}\label{sec:qw}

In this section, we revisit and establish new basic results on three types of state transfer on graphs.

\subsection{Peak state transfer}

The \textit{bounding matrix} of $X$ relative to $M$ is the matrix $B=B(X)$ such that for all $u,v\in V(X)$,
\begin{equation*}
B(X)_{u,v}=\sum_{\lambda\in\spec(M)}\big|(E_\lambda\big)_{u,v}\big|.
\end{equation*}
Since $E_\lambda$'s sum to $I$, we have $B(X)_{u,u}=1$ for all $u\in V(X)$. If $X$ is connected and $u,v$ are two arbitrary vertices in $X$, then $(E_0)_{u,v}>0$ when $M=L$ and $(E_\lambda)_{u,v}>0$ when $M=A$ where $\lambda$ is the largest eigenvalue of $A(X)$. In both cases, we see that if $X$ is connected, then $B(X)_{u,v}>0$ for all pairs of vertices $u$ and $v$ in $X$, i.e., $B(X)$ is a matrix with all positive entries. 

\begin{prop}
\label{prop:B+}
Let $M\in\{A,L\}$. If $X$ is a connected graph, then $B(X)$ has all positive entries.
\end{prop}

Since $B(X)$ is in general a nonnegative matrix, Proposition~\ref{prop:B+} implies that if $B(X)_{u,v}=0$ for some $u,v\in V(X)$, then $X$ is a disconnected graph and $u,v$ belong to different components of $X$. In particular, this observation implies that if $X$ is disconnected, then $B(X)$ is permutation-similar to a block diagonal matrix, where each block is the bounding matrix (with all positive entries) of a component of $X$.

The following result establishes a lower bound on the entries of the bounding matrix.

\begin{prop}
\label{prop:Buv}
Let $X$ be a connected graph on $n$ vertices. If (i) $M=A$ and $X$ is regular, or (ii) $M=L$, then $B(X)_{u,v}\geq \frac{2}{n}$ for all $u,v\in V(X)$. Moreover, if $u\neq v$, then equality is attained in complete graphs.
\end{prop}

\begin{proof}
Suppose that (i) or (ii) holds. So, $E_k=\frac{1}{n}\mathbf{J}$ when $M=A$ and $X$ is $k$-regular, and $E_0=\frac{1}{n}\mathbf{J}$ when $M=L$. Since $B(X)_{u,u}=1$ for all $u\in V(X)$, it suffices to prove the statement when $u\neq v$. Since $X$ is connected, $M$ has at least two distinct eigenvalues. Set $\theta=k$ when $M=A$ and $X$ is $k$-regular, and $\theta=0$ otherwise so that $(E_\theta)_{u,u}=\frac{1}{n}$. Because $E_\lambda$'s sum to $I$ and $u\neq v$, we have $\sum_{\lambda\in\spec(M)\backslash\{\theta\}}(E_\lambda)_{u,v}=-\frac{1}{n}.$
Thus, there exists $\varnothing\neq S\subseteq \spec(M)\backslash\{\theta\}$ such that $(E_\lambda)_{u,v}<0$ for all $\lambda\in S$ and $\sum_{\lambda\in S}(E_\lambda)_{u,v}\leq -\frac{1}{n}$. So $\big|\sum_{\lambda\in S}(E_\lambda)_{u,v}
\big|\geq\frac{1}{n}$. Thus, $B(X)_{u,v}\geq (E_\theta)_{u,u}+\big| \sum_{\lambda\in S}(E_\lambda)_{u,v}\big|= \frac{2}{n}$, with equality when $X=K_n$. 
\end{proof}

We now define peak state transfer. Using the fact that $U_X(t)=\sum_{\lambda\in\spec(M)}e^{it\lambda} (E_\lambda)$, we get
\begin{equation}
\label{eq:bounding}
|U_X(t)_{u,v}|\leq B(X)_{u,v}\quad \text{for all $t\in\R$}.
\end{equation}
We say that \textit{peak state transfer} occurs between two vertices $u$ and $v$ in $X$ a time $\tau$ if
\begin{equation}
\label{eq:peakst}
U(\tau)_{u,v}=\gamma B(X)_{u,v}
\end{equation}
for some unit $\gamma\in\C$. Peak state transfer was introduced by Coutinho, Guo, and Schmeits in \cite{coutinho2025peak}. 

We say that two vertices $u$ and $v$ in $X$  admit \textit{perfect state transfer} (PST) at time $\tau>0$ if peak state transfer occurs between them in $X$ at time $\tau$ and $B(X)_{u,v}=1$. Equivalently, $|U(\tau)_{u,v}|=1$. We say that $u$ is \textit{periodic} in $X$ if $|U(\tau)_{u,u}|=1$ for some $\tau>0$. It is known that if PST occurs between $u$ and $v$ at time $\tau$, then each of them is periodic at time $2\tau$ \cite[Section 3]{Godsil2012a}. A set $\mathcal{S}\subset \mathbb{R}$ with at least two elements satisfies the \textit{ratio condition} if $\frac{\lambda-\theta}{\mu-\eta}\in\Q$ for any $\lambda,\theta,\mu,\eta\in \mathcal{S}$ with $\mu\neq \eta$. It is known that the vertex $u$ is periodic in $X$ if and only if  $\supp_{u}$ satisfies the \textit{ratio condition} \cite{Godsil2011}. We say that vertex $u$ is \textit{almost periodic} if $\lim_{m\rightarrow\infty}|U(\tau_m)_{u,u}|=1$ for some sequence $\{\tau_m\}$. It is known that every vertex in a graph is almost periodic \cite[Theorem 2.7.2]{Coutinho2021}, and so $\sup_{t>0}|U(t)_{u,u}|=1$ for all $u\in V(X)$ in any graph $X$.

\begin{rem}
It is known that PST is a monogamous relation, i.e., if PST occurs between $u$ and $v$, and also between $u$ and $w$, then $v=w$ \cite{Kay2011}. However, peak state transfer is not monogamous in general. For example, complete graphs admit peak state transfer between any pair of vertices (see Example~\ref{ex:knnn}).
\end{rem}

\subsection{Pretty good state transfer}
Peak state transfer is a relaxation of PST. Another relaxation of PST is \textit{pretty good state transfer} (PGST), which happens between $u$ and $v$ if there exist a sequence $\{\tau_m\}$ and a unit complex number $\gamma$ such that 
\begin{equation}
\label{eq:pgst}
\lim_{m\rightarrow\infty}U(\tau_m)_{u,v}=\gamma.
\end{equation}
Note that for all $u\in V(X)$, $0\in\supp_u$ if $M=L$ and $\lambda\in \supp_u$ otherwise, where $\lambda$ is the largest eigenvalue of $A$. Thus, $\gamma=1$ in (\ref{eq:peakst}) and (\ref{eq:pgst}) if $M=L$. Otherwise, $\gamma=e^{i\tau \lambda}$ in (\ref{eq:peakst}) and $\gamma=\lim_{m\rightarrow\infty}e^{i\tau_m \lambda}$ in (\ref{eq:pgst}).% , where $\lambda$ is the largest eigenvalue of $A$.

Vertices $u$ and $v$ are \textit{strongly cospectral} if $E_\lambda\be_u=\pm E_\lambda\be_v\neq \zero$, in which case we may partition $\supp_u$ as 
\begin{center}
$\supp_{u,v}^+=\{\lambda\in \supp_u:E_\lambda\be_u=E_\lambda\be_v\}\quad $ and $\quad \supp_{u,v}^-=\{\lambda\in \supp_u:E_\lambda\be_u=-E_\lambda\be_v\}$.
\end{center}
%Strong cospectrality is a requirement for PGST. 

A useful tool for establishing the occurrence of PGST is Kronecker's Approximation Theorem \cite{apostol2012modular}.

\begin{theorem}
\label{thm:kron}
If $\alpha_1,\ldots,\alpha_r$ are arbitrary real numbers and $1,\lambda_1,\ldots,\lambda_r$ are real numbers linearly independent over $\mathbb{Q}$, then for $\epsilon>0$, there exist integers $q,p_1,\ldots,p_r$ such that for all $j\in\{1,\ldots,r\}$,
\begin{center}
$|q\lambda_j-p_j-\alpha_j|<\epsilon$.
\end{center}
\end{theorem}

In later sections, we will need examples of graphs that admit PGST relative to a sequence $\{\tau_m\}\subseteq 2\pi\Z$. This motivates us to provide sufficient conditions for this property to occur. 

\begin{theorem}
\label{thm:pgst}
Suppose $u$ and $v$ are strongly cospectral vertices in $X$.
\begin{enumerate}
\item If $\supp_{u}$ is a linearly independent set over $\Q$, then PGST occurs between $u$ and $v$ in $X$.
\item If $M=A$, $X$ is bipartite and $\supp_{u}\cap\R^+$ is a linearly independent set over $\Q$, then PGST occurs between $u$ and $v$.
\end{enumerate}
Moreover, for any rational number $c>0$, PGST occurs between $u$ and $v$  relative to a sequence $\{\tau_m\}\subseteq (\frac{2\pi}{c})\Z$ in 1 or 2 when $u$ and $v$ are in the same parts. In particular, we may choose $\{\tau_m\}\subseteq 2\pi\Z$. 
\end{theorem}

\begin{proof}
Let $u$ and $v$ be strongly cospectral in $X$ and $c>0$ be rational. From \cite{Banchi2017}, PGST occurs between $u$ and $v$ if and only if for each $\epsilon>0$, there exists $\tau\in\R$ such that for all $\lambda,\eta\in\supp_{u,v}^+$ and $\theta\in\supp_{u,v}^-$,
\begin{equation}
\label{eq:pgstchar}
|\tau (\lambda-\eta)|<\epsilon\pmod{2\pi}\quad \text{and}\quad |\tau (\lambda-\theta)-\pi|<\epsilon\pmod{2\pi}
\end{equation}

To prove 1, suppose that $\supp_{u}$ is a linearly independent set over $\Q$ and fix $\lambda\in\supp_{u,v}^+$. Then the set $\{\frac{1}{c}(\lambda-\mu):\mu\in \supp_u\backslash\{\lambda\}\}$ is also linearly independent over $\Q$. Thus, for each $\epsilon>0$, Theorem~\ref{thm:kron} implies that there exist integers $q,p_\lambda,p_\theta$ such that $\left|\frac{q}{c}(\lambda-\eta)-p_\eta\right|<\frac{\epsilon}{2\pi}$ and $\left|\frac{q}{c}(\lambda-\theta)-p_\theta-\frac{1}{2}\right|<\frac{\epsilon}{2\pi}$ both hold modulo $2\pi$
for all $\eta\in\supp_{u,v}^+$ and $\theta\in\supp_{u,v}^-$. Equivalently, 
\begin{center}
$\left|(2\pi q/c)(\lambda-\eta)- 2\pi p_\lambda\right|<\epsilon\quad $ and $\quad \left|(2\pi q/c)(\lambda-\theta)-2\pi p_\theta-\pi\right|<\epsilon$
\end{center}
modulo $2\pi$. Thus, for each $\epsilon>0$, (\ref{eq:pgstchar}) has a solution $\tau=2\pi q/c$ for some integer $q$, implying that PGST occurs between $u$ and $v$ in $X$ relative to a sequence $\{\tau_m\}\subseteq (2\pi/c)\Z$. This proves 1.

To prove 2, let $M=A$, $X$ be bipartite and $\supp_{u}\cap\R^+$ be a linearly independent set over $\Q$. Define $S^+=\supp_{u,v}^+\cap \R^+$ and $S^-=\supp_{u,v}^-\cap \R^+$. First, let $u$ and $v$ be in the same parts in $X$. As $X$ is bipartite, the form of eigenvectors of $\lambda$ and $-\lambda$ implies that $\supp_{u,v}^+$ and $\supp_{u,v}^-$ are closed under multiplication by $-1$. So, if $\lambda\in S^+$ is fixed, then (\ref{eq:pgstchar}) is equivalent to
\begin{equation}
\label{eq:pgstchar1}
|\tau(\lambda-\eta)|<\epsilon,\quad 
|\tau(2\lambda)|<\epsilon\quad \text{and}\quad |\tau(\lambda-\theta)-\pi|<\epsilon 
\end{equation}
modulo $2\pi$ for all $\eta\in S^+$ and $\theta\in S^-$. We claim that $\{\lambda-\eta:\lambda\neq\eta\in S^+\}\cup \{\lambda-\theta:\theta\in S^-\}\cup\{2\lambda\}$ is also a linearly independent set over $\Q$. To see this, if $a_\eta,b_\theta,c$ are integers such that
\begin{equation*}
\sum_{\lambda\neq \eta\in S^+}a_\eta(\lambda-\eta)+ \sum_{\theta\in S^-}b_\theta(\lambda-\theta)+c(2\lambda)=0,
\end{equation*}
then rewriting this equation and and using the fact that $\supp_{u}\cap\R^+$ is a linearly independent set over $\Q$ yields $a_\eta=0$ for all $\eta\neq \lambda$, $b_\theta=0$ for all $\theta$, and $c=0$. This proves the claim. Applying the same argument as in 1, there exist integers $q,p_\lambda,p_\theta$ such that modulo $2\pi$, we have
\begin{center}
$\left|(2\pi q/c)(\lambda-\eta)- 2\pi p_\lambda\right|<\epsilon,\quad $ $\left|(2\pi q/c)(2\lambda)\right|<\epsilon\quad$ and $\quad\left|(2\pi q/c)(\lambda-\theta)-2\pi p_\theta-\pi\right|<\epsilon$
\end{center}
for all $\eta\in S^+$ and $\theta\in S^-$. Thus, for each $\epsilon>0$, (\ref{eq:pgstchar1}) has a solution $\tau=\frac{2\pi q}{c}$ for some integer $q$. Since (\ref{eq:pgstchar1}) is equivalent to (\ref{eq:pgstchar}), $\tau=\frac{2\pi q}{c}$ is also a solution to (\ref{eq:pgstchar1}), implying that PGST occurs between $u$ and $v$ in $X$ relative to a sequence $\{\tau_m\}\subseteq (\frac{2\pi}{c})\Z$. Next, suppose $u$ and $v$ are in the same parts in $X$. As $X$ is bipartite, the form of eigenvectors of $\lambda$ and $-\lambda$ implies that
$\lambda\in \supp_{u,v}^+$ if and only of $-\lambda\in \supp_{u,v}^-$. Thus, for a fixed $\lambda\in S^+$, (\ref{eq:pgstchar}) is equivalent to
\begin{equation*}
|\tau (\lambda-\eta)|<\epsilon,\quad |\tau (\lambda+\theta)|<\epsilon,\quad |\tau (\lambda+\eta)-\pi|<\epsilon,\quad \text{and}\quad |\tau (\lambda-\theta)-\pi|<\epsilon.
\end{equation*}
modulo $2\pi$ for all $\eta\in S^+$ and $\theta\in S^-$. Now, if $a_\eta,b_\theta,c_\eta,d_\theta$ are integers such that
\begin{center}
$\displaystyle\sum_{\eta\in S^+\backslash\{\lambda\}} a_\eta(\lambda-\eta)+\displaystyle\sum_{\theta\in S^-\backslash\{\lambda\}} b_\theta(\lambda+\theta)+\displaystyle\sum_{\eta\in S^+\backslash\{\lambda\}} c_\eta(\lambda+\eta)+\displaystyle\sum_{\theta\in S^-\backslash\{\lambda\}} d_\theta(\lambda-\theta)=0$, 
\end{center}
then, as $\supp_{u}\cap\R^+$ is linearly independent over $\Q$, $a_\eta=c_\eta$, $b_\theta=d_\theta$ and
$2\displaystyle\sum_{\eta\in S^+\backslash\{\lambda\}} a_\eta + 2\displaystyle\sum_{\theta\in S^-\backslash\{\lambda\}} b_\theta=0$. Applying \cite[Lemma 1]{kempton2017pretty}, we obtain PGST between $u$ and $v$. 
\end{proof}

\begin{rem}
Note that Theorem~\ref{thm:pgst}(1) is well-known, but the novelty of this result is the fact that the sequence of times may be chosen to lie in $2\pi \Z$.
\end{rem}

\begin{ex}
\label{ex:path1}
Let $M=A$. From \cite{Godsil2012}, $P_n$ has PGST between end vertices if and only if the set $\spec(A)\cap\R^+$ is linearly independent over $\Q$. Equivalently, $n+1\in\{2^t,p,2p\}$ where $t\geq 1$ is an integer and $p\geq 1$ is a prime, in which case PGST in $P_n$ occurs between $j$ and $n+1-j$ for all $j\in V(P_n)$. In particular, if $n+1\in\{2^t,2p\}$, then $j$ and $n+1-j$ belong to the same parts in $P_n$, and so PGST occurs between them relative to $\{\tau_m\}\subseteq 2\pi\Z$ by Theorem~\ref{thm:pgst}. 
%We note however that if $n+1\in\{2^t,2p\}$ and $k+1\in\{2^{t'},2p'\}$ with $n\neq k$, then $P_n$ and $P_k$ have PGST between $j$ and $n+1-j$ for all $j\in V(P_n)$ relative to $\{\tau_m\},\{\tau_\ell\}\subseteq 2\pi\Z$ respectively, but we do not know if $\{\tau_m\},\{\tau_\ell\}$ contain a common subsequence $\{\tau_r\}$ (if such a subsequence exists, then $P_n$ and $P_k$ have PGST relative to the same sequence $\{\tau_r\}$).
\end{ex}

\begin{ex}
\label{ex:path2}
Let $M=A$. We show that $P_4$ admits PGST between end vertices $u$ and $v$ relative to $\{\tau_m\}\subseteq (2\Z+1)\pi$. Note that $u$ and $v$ belong to different parts, and the positive eigenvalues of $P_4$ are linearly independent over $\Q$. PGST occurs between $u$ and $v$ in $P_4$ if and only if
\begin{center}
$e^{i\tau(\frac{1+\sqrt{5}}{2})}\approx e^{i\tau(\frac{1-\sqrt{5}}{2})}\approx-e^{i\tau(\frac{-1+\sqrt{5}}{2})}\approx-e^{-i\tau(\frac{1+\sqrt{5}}{2})}$.
\end{center}
Equivalently, $e^{i\tau(1+\sqrt{5})}\approx e^{i\tau}\approx-e^{i\tau\sqrt{5}}\approx-1$, which holds if and only if for each $\epsilon>0$, $|t-\pi|<\epsilon$ and $|\sqrt{5}t|<\epsilon$ both have a solution modulo $2\pi$. Since $1$ and $\sqrt{5}$ are linearly independent over $\Q$, Theorem~\ref{thm:kron} yields a solution to this system. So PGST occurs between $u$ and $v$ in $P_4$ relative to $\{\tau_m\}\subseteq (2\Z+1)\pi$.
\end{ex}

We also mention some results due to Pal. %(see Theorem 2.7 and Corollary 2.9 in \cite{pal2021quantum}, resp).
A \textit{circulant graph} is a Cayley graph $\operatorname{Cay}(\Z_n,S)$. %A \textit{gcd-set} in $\Z_n$ is a set $\bigcup_{d\in D}S_n(d)$, where $d$ is
For a divisor $1\leq d<n$ of $n$, %$D$ is a nonempty set of proper divisors of $n$, and 
we let $S_n(d)=\{x\in\Z_n:\operatorname{gcd}(x,n)=d\}$.

\begin{ex}
\label{ex:Cn}
Let $n=2^k$ and $\operatorname{Cay}(\Z_n,S)$ be a non-integral circulant graph. If $d$ is the least among all the divisors of $n$ such that $S\cap S_n(d)\subsetneq S_n(d)$ and $|S\cap S_n(d)|\equiv 2 \pmod{4}$, then $\operatorname{Cay}(\Z_n,S)$ and $C_{n/d}$ admit PGST between $u$ and $v=u+\frac{n}{2}$ relative to the same sequence $\{\tau_m\}\subseteq 2\pi\Z$ \cite[Theorem 2.7]{pal2021quantum}. 
\end{ex}

If $n=2^k$, then $\operatorname{Cay}(\Z_n,\{1,-1\})$, which is the cycle $C_n$, admits PGST relative to some $\{\tau_m\}\subseteq 2\pi\Z$ by the above example. Since the positive eigenvalues of $C_n$ with $n=2^k$ are also linearly independent over $\Q$ \cite{Pal2017}, we recover the preceding statement using Theorem~\ref{thm:pgst}(2).

\begin{ex}
\label{ex:Cn1}
Let $n=2^k$. If each divisor $d$ of $n$ satisfies $|S\cap S_n(d)|\equiv 0 \pmod{4}$, then $\operatorname{Cay}(\Z_n,S\cup\{\frac{n}{2}\})$ and $\operatorname{Cay}(\Z_n,S\cup\{\frac{n}{4},\frac{3n}{4}\})$ admit PGST between $u$ and $v=u+\frac{n}{2}$ relative to the same sequence $\{\tau_m\}\subseteq (2\Z+1)\frac{\pi}{2}$ \cite[Corollary 2.9]{pal2021quantum}.
\end{ex}

\subsection{Vertex sedentariness}
A vertex $u$ is \textit{$C$-sedentary} in $X$ if $\inf_{t>0}\lvert U_X(t)_{u,u}\rvert\geq C$ for some constant $0<C\leq 1$. A vertex $u$ is \textit{not sedentary} in $X$ if $\inf_{t>0}\lvert U_X(t)_{u,u}\rvert=0$. If $C$ is not important, then we say sedentary in lieu of $C$-sedentary.

The following sufficient condition for sedentariness will be useful later \cite[Lemma 9]{monterde2026new}. 

\begin{theorem}
\label{thm:sed}
If $(E_{\theta})_{u,u}>\frac{1}{2}$, 
then $u$ is sedentary in $X$ with $|U_X(t)_{u,u}|\geq 2(E_{\theta})_{u,u}-1$. 
\end{theorem}

\begin{rem}
Sedentariness and PGST are known to be mutually exclusive \cite{Monterde2023}. That is, if PGST occurs between $u$ and $v$, then none of them are sedentary, and if $u$ is sedentary, then it cannot be involved in PGST. However, vertex sedentariness and peak state transfer are not necessarily mutually exclusive, i.e., it is possible for a vertex to be sedentary and also be involved in peak state transfer.
\end{rem}

We now establish a connection between sedentariness and peak state transfer.

\begin{prop}
\label{prop:sedpeak}
Let $u,v$ be two vertices in $X$. If vertex $u$ is $C$-sedentary, then $|U_X(t)_{u,v}|\leq \sqrt{1-C^2}$ for all $t\in\R$. Additionally, if peak state transfer occurs between $u$ and $v$ in $X$ at time $\tau$, then for all $t\in\R$,
\begin{center}
$|U_X(t)_{u,v}|\leq |U_X(\tau)_{u,v}|=B(X)_{u,v}\leq \sqrt{1-C^2}$.
\end{center}
\end{prop}

\begin{proof}
By the unitarity of $U_X(t)$, we get $|U_X(t)_{u,u}|^2+|U_X(t)_{u,v}|^2\leq 1$ for all $t\in\R$. Since $u$ is $C$-sedentary, we have $|U_X(t)_{u,u}|\geq C$ for all $t\in\R$. So, $|U_X(t)_{u,v}|^2\leq 1-|U_X(t)_{u,u}|^2\leq 1-C^2$ for all $t\in\R$. Thus, $|U_X(t)_{u,v}|\leq \sqrt{1-C^2}$ for all $t\in\R$. As peak state transfer also occurs between $u$ and $v$ at time $\tau$, $B(X)_{u,v}$ is attained by $|U_X(t)_{u,v}|$ at time $\tau$. Combining these observations yields the result.
\end{proof}

We now demonstrate that $\sqrt{1-C^2}$ is an asymptotically tight upper bound for $B(X)_{u,v}$. 

\begin{ex}
\label{ex:sedpeakst}
For $n\geq 1$, let $G_n$ be the graph obtained from $K_{2,n}$ by attaching a pendent vertex to each vertex in the part of size $n$. Let $u$ and $v$ be the twins in $G_n$ as depicted in Fig.~\ref{fig:Gn}. In \cite[Section~6.2]{coutinho2025peak}, it was shown that peak state transfer occurs between $u$ and $v$ in $G_n$ with $B(G_n)_{u,v}=\frac{2n}{2n+1}$. It was also shown that the spectral idempotent $E_0$ for $G_n$ satisfies $(E_0)_{u,u}=(E_0)_{v,v}=\frac{n+1}{2n+1}>\frac{1}{2}$. By Theorem
\ref{thm:sed}, 
$u$ and $v$ are $C(n)$-sedentary in $G_n$ for all $n\geq 1$, where $C(n)=\frac{1}{2n+1}$. As $n\rightarrow\infty$, we get
\begin{center}
    $1\leq \frac{\sqrt{1-C(n)^2}}{B(G_n)_{u,v}}=\frac{\sqrt{1-1/(2n+1)^2}}{2n/(2n+1)}=%\frac{2\sqrt{n(n+1)}/(2n+1)}{2n/(2n+1)}=
\frac{\sqrt{n(n+1)}}{n}\rightarrow 1.$
\end{center}
\end{ex}

For most recent work on sedentariness, see \cite{meagher2026sedentary}.

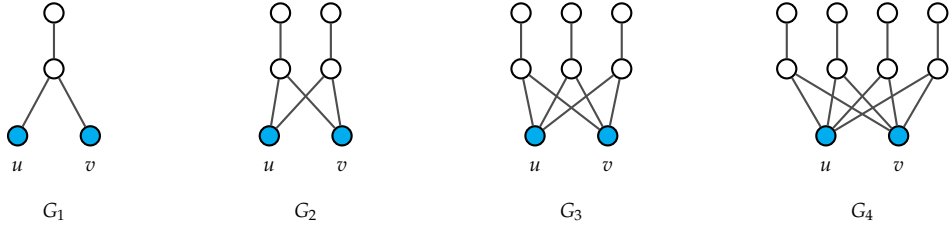
\begin{figure}\label{fig:Gn}
	\begin{center}
		\begin{tikzpicture}[scale=.37,auto=left]
		\tikzstyle{every node}=[circle, thick, fill=white, scale=0.7]
		
		% G_n: n copies of K_2 (t_i -- m_i), plus twins u, v adjacent to every m_i
		\foreach \n/\c in {1/0, 2/9, 3/18.5, 4/28.9}{
			\node[draw,fill=cyan] (u\n) at (\c-1.3,-1.6) {};
			\node[draw,fill=cyan] (v\n) at (\c+1.3,-1.6) {};
			\node at (\c-1.3,-2.7) {$u$};
			\node at (\c+1.3,-2.7) {$v$};
			\node at (\c,-4.3) {$G_{\n}$};
			\foreach \i in {1,...,\n}{
				\pgfmathsetmacro{\x}{\c+1.8*(\i-(\n+1)/2)}
				\node[draw] (m\n-\i) at (\x,0.8) {};
				\node[draw] (t\n-\i) at (\x,2.8) {};
				\draw [thick, black!70] (t\n-\i)--(m\n-\i);
				\draw [thick, black!70] (u\n)--(m\n-\i)--(v\n);
			}
		}
		\end{tikzpicture}
		\vspace{-10pt}
	\end{center}	
	\caption{The graphs $G_n$ with the twin vertices $u$ and $v$ marked blue, both of which are $C(n)$-sedentary.}
\end{figure}

\section{Joins}
\label{sec:join}

Let $X$ and $Y$ be graphs such that $|V(X)|=n_1$, $|V(Y)|=n_2$, and $n=n_1+n_2$. In this section, we study the behavior of PGST and peak state transfer under the join operation. We follow the discussion in \cite{monterde2025quantum}. For further investigation on quantum walks on join graphs, see \cite{kirkland2023quantum}.

The \textit{join} of $X$ and $Y$, denoted $X\lor Y$, is the graph with vertex set $V(X)\cup V(Y)$ such that
\begin{equation*}
\label{eq:adjlapjoin}
    A(X\lor Y) =    \begin{bmatrix}
                        A(X) & \mathbf{J} \\
                        \mathbf{J} & A(Y)
                    \end{bmatrix}\quad \text{and} \quad L(X\lor Y) =    \begin{bmatrix}
                        L(X)+n_2I & -\mathbf{J}\\
                        -\mathbf{J}& L(Y)+n_1I
                    \end{bmatrix}.
\end{equation*}

Suppose $M=L$. If $X$ and $Y$ are both connected, then we may write
\begin{equation}
\label{eq:joinL}
U_{X\vee Y}(t)=\frac{1}{n}J+e^{i tn}E_{n}
+\sum_{\lambda>0}e^{it(\lambda+n_2)}\left[\begin{array}{cc} E_{\lambda}&O \\ O&O\end{array} \right]+\sum_{\mu>0}e^{it(\mu+n_1)}\left[\begin{array}{cc} O&O \\ O&F_{\mu}\end{array} \right], 
\end{equation}
where $E_{n}=\frac{1}{n_1n_2n}\left[ \begin{array}{ccccc} n_2^2J&-n_1n_2J \\ -n_1n_2J&n_1^2J\end{array} \right]$ and $F_{\mu}$ is the spectral idempotent of $\mu\in\spec(L(X_2))\backslash\{0\}$. We include $e^{itn_2}\left[\begin{array}{cc} E_{0}-\frac{1}{n_1}J&O\\ O&O\end{array} \right]$ (resp., $e^{i tn_1}\left[\begin{array}{cc} O&O \\ O&F_{0}-\frac{1}{n_2}J\end{array} \right]$) as a term in the third (resp., fourth) summand in (\ref{eq:joinL}) whenever $X$ (resp., $Y$) is disconnected. 

For the case $M=A$, we further assume that $X$ and $Y$ are $k$-regular and $\ell$-regular, respectively. If $X$ and $Y$ are both connected, then we may write 
\begin{equation}
\label{eq:joinA}
U_{X\vee Y}(t)=e^{i t\lambda^+}E_{\lambda^+}+e^{i t\lambda^-}E_{\lambda^-}+\sum_{\lambda<k}e^{i t\lambda}\left[\begin{array}{cc} E_{\lambda}&O \\ O&O\end{array} \right]+\sum_{\mu<\ell}e^{i t\mu}\left[\begin{array}{cc} O&O \\ O&F_{\mu}\end{array} \right],     
\end{equation}
where $\lambda^{\pm} = \frac{1}{2}\left(k+\ell\pm \sqrt{\Delta}\right)$, $E_{\lambda^+}=\frac{1}{n_1\sqrt{\Delta}(k-\lambda^-)}\bu\bu^\top$, $E_{\lambda^-}=\frac{1}{n_1\sqrt{\Delta}(\lambda^+-k)}\bv\bv^\top$, $\Delta=(k-\ell)^2+4n_1n_2$, and the vectors $\bu:=\left[ \begin{array}{ccccc} (k-\lambda^{-})\one_{n_1} \\ n_1\one_{n_2}\end{array} \right]$ and $\bv:=\left[ \begin{array}{ccccc} (k-\lambda^{+})\one_{n_1} \\ n_1\one_{n_2}\end{array} \right]$ are eigenvectors associated with 
$\lambda^{+}$ and $\lambda^{-}$, respectively. 
We include the term $e^{i tk}\left[\begin{array}{cc} E_{k}-\frac{1}{n_1}J&O \\ O&O\end{array} \right]$ (resp., $e^{i t\ell}\left[\begin{array}{cc} O&O \\ O&F_{\ell}-\frac{1}{n_2}J\end{array} \right]$) in the third (resp., fourth) summand in equation (\ref{eq:joinA}) whenever $X$ (resp., $Y$) is disconnected.

\begin{lemma}
\label{lem:joinbound}
Let $M\in\{A,L\}$ and consider  $X\vee Y$, where we assume that $X,Y$ are $k$- and $\ell$-regular graphs, respectively, whenever $M=A$. If $u,v\in V(X)$, then $B(X\vee Y)_{u,v}=B(X)_{u,v}$. Otherwise, $B(X\vee Y)_{u,v}=\frac{2}{n}$ if $M=L$ and $B(X\vee Y)_{u,v}=\frac{2}{\sqrt{\Delta}}$ otherwise.  
\end{lemma}

\begin{proof}
Let $u,v\in V(X)$. If $M=L$, then (\ref{eq:joinL}) yields
$B(X\vee Y)_{u,v}=\frac{1}{n}+\frac{n_2}{n_1n}+B(X)_{u,v}-\frac{1}{n_1}=B(X)_{u,v}$. If $M=A$, then (\ref{eq:joinA}) yields $B(X\vee Y)_{u,v}=\frac{k-\lambda^-}{n_1\sqrt{\Delta}}-\frac{k-\lambda^+}{n_1\sqrt{\Delta}}+B(X)_{u,v}-\frac{1}{n_1}=B(X)_{u,v}$. Combining these two cases proves the case $u,v\in V(X)$, while the other case is immediate from (\ref{eq:joinL}) and (\ref{eq:joinA}).
\end{proof}

We now provide sufficient conditions such that PGST and peak state transfer in a graph are preserved in the join (with equal transfer probabilities for the case of peak state transfer).

\begin{theorem}
\label{thm:peakstjoin}
Let $M\in\{A,L\}$ and consider $X\vee Y$, where we assume that $X,Y$ are $k$- and $\ell$-regular graphs, respectively, whenever $M=A$. If $u,v\in V(X\vee Y)$, then the following hold.
\begin{enumerate}
\item Suppose that peak state transfer occurs between $u$ and $v$ in $X$ at time $\tau$. If $M=A$, then peak state transfer occurs between $u$ and $v$ in $X\vee Y$ at time $\tau$ if and only if $e^{i\tau \lambda^+}=e^{i\tau \lambda^-}=e^{i\tau k}$ (equivalently, $\{k,\lambda^+,\lambda^-\}$ satisfies the ratio condition and $\lambda^+,\lambda^-$ are integers if $k\neq \ell$), in which case $U_{X\vee Y}(\tau)_{u,v}=e^{i\tau k}B(X)_{u,v}$. If $M=L$, then peak state transfer occurs between $u$ and $v$ in $X\vee Y$ at time $\tau$ if and only if $\tau=\frac{2r\pi}{\operatorname{gcd}(n_1,n_2)}$ for some integer $r$, in which case $U_{X\vee Y}(\tau)_{u,v}= B(X)_{u,v}$.
\item Suppose PGST occurs between $u$ and $v$ in $X$ relative to $\{\tau_m\}$. If $M=A$, then PGST occurs between $u$ and $v$ in $X\vee Y$ at time $\tau$ if and only if $\lim_{m\rightarrow\infty}e^{i\tau_m\lambda^+}=\lim_{m\rightarrow\infty}e^{i\tau_m\lambda^-}=\lim_{m\rightarrow\infty}e^{i\tau_mk}=\gamma$, in which case $\lim_{m\rightarrow\infty}U_{X\vee Y}(\tau_m)_{u,v}=\gamma$. If $M=L$, then PGST occurs between $u$ and $v$ in $X\vee Y$ at time $\tau$ if and only if $\lim_{m\rightarrow\infty}e^{i\tau_m n}=\lim_{m\rightarrow\infty}e^{i\tau_m n_2}=1$, in which case $\lim_{m\rightarrow\infty}U_{X\vee Y}(\tau_m)_{u,v}=1$.
\end{enumerate}    
\end{theorem}

\begin{proof}
To prove 1, let $u,v\in V(X)$ and suppose that peak state transfer occurs between $u$ and $v$ in $X$ at time $\tau$. Then $U_X(\tau)_{u,v}=\gamma B(X)_{u,v}$ for some unit complex number $\gamma$. For $M=A$, we have $\gamma=e^{i\tau k}$. Since $U_X(\tau)_{u,v}=e^{i\tau k}B(X)_{u,v}$, (\ref{eq:joinA}) implies that $|U_{X\vee Y}(\tau)_{u,v}|$ equals
\begin{center}
$\big|\frac{e^{i\tau\lambda^+}(k-\lambda^-)}{n_1\sqrt{\Delta}}-\frac{e^{i\tau\lambda^-}(k-\lambda^+)}{n_1\sqrt{\Delta}}+e^{i\tau k}\big(B(X)_{u,v}-\frac{1}{n_1}\big)\big|
\leq \frac{k-\lambda^-}{n_1\sqrt{\Delta}}-\frac{k-\lambda^+}{n_1\sqrt{\Delta}}+B(X)_{u,v}-\frac{1}{n_1},$
\end{center}
where the right-hand side equals $B(X)_{u,v}$. Thus, $U_{X\vee Y}(\tau)_{u,v}=B(X)_{u,v}$ if and only if $e^{i\tau \lambda^+}=e^{i\tau \lambda^-}=e^{i\tau k}$. Equivalently, $\{k,\lambda^+,\lambda^-\}$ satisfies the ratio condition, in which case $\lambda^+$ and $\lambda^-$ are integers if $k\neq \ell$. 
Now, if $M=L$, then $\gamma=1$. Since $U_X(\tau)_{u,v}= B(X)_{u,v}$, (\ref{eq:joinL}) implies that $|U_{X\vee Y}(\tau)_{u,v}|$ equals
\begin{center}
$\big|\frac{1}{n}+\frac{n_2e^{i\tau n}}{n_1n}+e^{i\tau n_2}\big(B(X)_{u,v}-\frac{1}{n_1}\big)\big|\leq \frac{1}{n}+\frac{n_2}{n_1n}+B(X)_{u,v}-\frac{1}{n_1}=B(X)_{u,v}$.
\end{center}
Thus, $U_{X\vee Y}(\tau)_{u,v}=B(X)_{u,v}$ if and only if $e^{i\tau n}=e^{i\tau n_2}=1$, i.e. $\tau$ is an integer multiple of $\frac{2\pi}{\operatorname{gcd}(n_1,n_2)}$.
Combining these two cases with Lemma~\ref{lem:joinbound} yields 1. The proof of 2 follows from 1 by replacing $e^{i\tau k}$, $e^{i\tau \lambda^+}$ and $e^{i\tau \lambda^-}$ with $\lim_{m\rightarrow\infty}e^{i\tau_mk}$, $\lim_{m\rightarrow\infty}e^{i\tau_m\lambda^+}$ and $\lim_{m\rightarrow\infty}e^{i\tau_m\lambda^-}$ respectively when $M=A$, and replacing $e^{i\tau n}$ and $e^{i\tau n_2}$ with $\lim_{m\rightarrow\infty}e^{i\tau_mn}$ and $\lim_{m\rightarrow\infty}e^{i\tau_mn_2}$ respectively when $M=L$.
\end{proof}

\begin{theorem}
\label{thm:peakstjoin1}
Let $u\in V(X)$ and $v\in V(Y)$.
With the assumption in Theorem~\ref{thm:peakstjoin}, we have $U_{X\vee Y}(t)_{u,v}=\frac{e^{it\lambda^+}}{\sqrt{\Delta}}(1-e^{-it\sqrt{\Delta}})$ if $M=A$ and $U_{X\vee Y}(t)_{u,v}=\frac{1}{n}(1-e^{itn})$, otherwise. Thus, peak state transfer occurs between $u$ and $v$ in $X\vee Y$, where $U(\frac{\pi}{\sqrt{\Delta}})_{u,v}=\frac{2e^{i\tau\lambda^+}}{\sqrt{\Delta}}$ if $M=A$ and $U(\frac{\pi}{n})_{u,v}=\frac{2}{n}$ if $M=L$. Moreover, PGST does not occur between $u$ and $v$ in $X\vee Y$.
\end{theorem}

\begin{proof}
This is immediate from immediate from (\ref{eq:joinA}) and (\ref{eq:joinL}).
\end{proof}

Using Theorem~\ref{thm:peakstjoin}(1), we produce new infinite families of graphs with peak state transfer.

\begin{ex}
Let $X$ be the Petersen graph, which is a 3-regular graph on 10 vertices. Let $u$ and $v$ be adjacent vertices in $X$. Then $X$ admits peak state transfer between $u$ and $v$ relative to $A$ at time $\tau=\pi$ with $|U_{X}(\tau)_{u,v}|=\frac{8}{15}$ \cite[Section 6.1]{coutinho2025peak}. Since $X$ is regular, this also holds relative to $L$. 
\begin{itemize}
\item Since $\tau=\pi$, Theorem~\ref{thm:peakstjoin}(1) yields peak state transfer in $X\vee Y$  between $u$ and $v$ relative to $L$ at time $\tau$ if and only if $Y$ has an even number of vertices.
\item Let $Y$ be an $\ell$-regular with $n$ vertices. Since $e^{i\tau k}=-1$ Theorem~\ref{thm:peakstjoin}(1) yields peak state transfer relative to $A$ between $u$ and $v$ in $X\vee Y$ at time $\tau$ if and only if $e^{i\tau \lambda^+}=e^{i\tau \lambda^-}=-1$. That is, $\lambda^{\pm}=\frac{1}{2}(3+\ell\pm\sqrt{(\ell-3)^2+40n})$ are odd integers. There are many such graphs $Y$, for example we may take $\ell=3$ and $n=10s^2$ for any integer $s\geq 1$, in which case $\lambda^{\pm}=3\pm 10s$ are both even.
\end{itemize}
\end{ex}

\begin{ex}
\label{ex:knnn}
Let $m\geq 2$ and $X$ be the balanced complete multipartite graph $K_{m,\ldots,m}$ with $k\geq 2$ parts. We may write $X=\bigvee_k O_m=O_m\vee \big(\bigvee_{k-1} O_m\big)$,
where $O_m$ is the empty graph on $m$ vertices and $\bigvee_{k-1} O_m$ is $m(k-2)$-regular with $m(k-1)$ vertices. Note that 
$\Delta=(0-(m(k-2))^2)+4m(m(k-1))=m^2k^2$. We apply Theorem~\ref{thm:peakstjoin1} with $M=A$ (this also holds for $M=L$ since $X$ is regular). If $u$ and $v$ are adjacent in $X$, 
then peak state transfer occurs between them at time $\tau=\frac{\pi}{mk}$ and $|U(\tau)_{u,v}|=\frac{2}{mk}$. In particular, if $m=2$ (resp., $m=1$), then $X$ is the cocktail party graph on $2k$ vertices (resp., $X=K_k$).
\end{ex}

\begin{ex}
\label{ex:kmn}
For all $m,n\geq 1$, the complete bipartite graph $K_{m,n}=O_m\vee O_n$ has peak state transfer between any two adjacent vertices at time $\tau=\frac{\pi}{m+n}$ if $M=L$ and $\tau=\frac{\pi}{2\sqrt{mn}}$ otherwise, by Theorem~\ref{thm:peakstjoin1}. 
\end{ex}

Next, we utilize Theorem~\ref{thm:peakstjoin}(2) to obtain new examples of pretty good state transfer.

\begin{ex}
\label{ex:pgst1}
Let $X$ be graph on $n_1$ vertices, where we assume that $X$ is $k$-regular when $M=A$. Suppose $X$ has PGST between $u$ and $v$ relative to $\{\tau_m\}\subseteq 2\pi\Z$ (see Theorem~\ref{thm:pgst} for sufficient conditions for this to occur). Since $\lim_{m\rightarrow\infty}e^{i\tau_m p}=1$ for any integer $p$, PGST occurs between $u$ and $v$ in $X\vee Y$ relative to $\{\tau_m\}$ for any graph $Y$ when $M=L$. Meanwhile, when $M=A$, if $\lambda^+,\lambda^-$ are integers, then PGST occurs between $u$ and $v$ in $X\vee Y$ relative to $\{\tau_m\}$. For example, if $X=\operatorname{Cay}(\Z_{2^r},S)$ in Example~\ref{ex:Cn}, then PGST occurs between $u$ and $v=u+2^{r-1}$ in $X$ relative to $\{\tau_m\}\subseteq 2\pi\Z$. So if $Y$ is $\ell$-regular with $n_2$ vertices and $(\ell-|S|)^2+2^{r+2}n_2$ is a perfect square, then $X\vee Y$ has PGST between $u$ and $v$ relative to $\{\tau_m\}$.
\end{ex}

\begin{ex}
Suppose the same hypothesis in Example~\ref{ex:pgst1} hold, except we assume  $\{\tau_m\}\subseteq (2\Z+1)\frac{\pi}{2}$. If $n,n_2\equiv d\pmod{4}$ where $d\in\{0,2\}$, then $\lim_{m\rightarrow\infty}e^{i\tau_mn}=\lim_{m\rightarrow\infty}e^{i\tau_mn_2}=\pm 1$, and so PGST occurs between $u$ and $v$ in $X\vee Y$ relative to $\{\tau_m\}$ when $M=L$. For the case $M=A$, if $\lambda^+,\lambda^-$ are integers and $k,\lambda^+,\lambda^-\equiv d\pmod{4}$, where $d\in\{0,2\}$, then $\lim_{m\rightarrow\infty}e^{i\tau_m\lambda^+}=\lim_{m\rightarrow\infty}e^{i\tau_m\lambda^-}=\lim_{m\rightarrow\infty}e^{i\tau_mk}=\pm 1$, and so PGST occurs between $u$ and $v$ in $X\vee Y$ relative to $\{\tau_m\}$. Note that $n_2$ and $\ell$ may be chosen appropriately to obtain $k,\lambda^+,\lambda^-\equiv d\pmod{4}$ (for instance, $n_2$ and $\ell$ may be chosen such that $k=\ell\equiv d\pmod{4}$ and $\sqrt{n_1n_2}\equiv 0\pmod{4}$). By way of example, we may take $X=\operatorname{Cay}(\Z_{n_1},S\cup\{\frac{n_1}{4},\frac{3n_1}{4}\})$ in Example~\ref{ex:Cn1}, where $n_1=2^r$ and $|S\cap S_{n_1}(d)|\equiv 0\pmod{4}$ (this happens if $S$ has four elements from each of $S_{n_1}(d)$ for all $1\leq d<\frac{n_1}{4}$). As $|S|\equiv 0\pmod{4}$, we get $k=|S\cup\{\frac{n_1}{4},\frac{3n_1}{4}\}|\equiv 2\pmod 4$. So, if $\lambda^+,\lambda^-\equiv 2\pmod{4}$, then PGST occurs between $u$ and $v=u+\frac{n}{2}$ in $X\vee Y$ relative to $\{\tau_m\}$. 
\end{ex}

\section{Cartesian products}
\label{sec:cartprod}

We now study the behavior of PGST and peak state transfer under Cartesian products.

The \textit{Cartesian product} of $X$ and $Y$, denoted $X\square Y$, is the graph with vertex set $V(X)\times V(Y)$ such that
\begin{equation*}
A(X\square Y)=A(X)\otimes I+I\otimes A(Y)\quad\text{and} \quad L(X\square Y)=L(X)\otimes I+I\otimes L(Y).
\end{equation*}
If $M\in\{A,L\}$, then it is known that
\begin{equation}
\label{eq:cartprod}
U_{X\square Y}(t)=U_X(t)\otimes U_Y(t).
\end{equation}
If $\mu$ and $\theta$ are eigenvalues of $M(X)$ and $M(Y)$ with associated eigenvectors $\bv$ and $\bw$ respectively, then $\mu+\theta$ is an eigenvalue of $M(X\square Y)$ with associated eigenvector $\bv\otimes \bw$. Therefore, for each $\lambda\in\spec(M(X\square Y))$, we may write the spectral idempotent $E_\lambda$ of $M(X\square Y)$ as
\begin{equation}
\label{eq:cartE}
E_\lambda=\sum_{{\mu\in\spec(M(X)),\ \theta\in\spec(M(Y))\atop \lambda=\mu+\theta} }E_{\mu}\otimes E_{\theta}.
\end{equation}
We now bound the entries of the bounding matrix of $X\square Y$.

\begin{prop}
\label{lem:cartbound}
Let $M\in\{A,L\}$ and consider $X\square Y$. If $u,v\in V(X)$ and $x,y\in V(Y)$, then
\begin{center}
$B(X\square Y)_{(u,x),(v,y)}\leq B(X)_{u,v}\cdot B(Y)_{x,y}.$
\end{center}
\end{prop}

\begin{proof}
For each $\lambda\in\spec(M(X\square Y))$, (\ref{eq:cartE}) and the triangle inequality give us
\begin{center}
$B(X\square Y)_{(u,x),(v,y)}=\displaystyle\sum_{\lambda\in\spec(M(X\square Y))}|(E_\lambda)_{(u,x),(v,y)}|\leq \sum_{\mu\in\spec(M(X))}\sum_{\theta\in\spec(M(Y))}|(E_{\mu})_{u,v}|\cdot |(E_{\theta})_{x,y}|.$
\end{center}
Since we may write the rightmost expression as $B(X)_{u,v}\cdot B(Y)_{x,y}$, we get the desired conclusion.
\end{proof}

\begin{theorem}
\label{thm:peakstcartprod}
Suppose $X$ admits peak state transfer between $u$ and $v$ at time $\tau$.
\begin{enumerate}
\item If $Y$ admits peak state transfer between $x$ and $y$ at time $\tau$, then $X\square Y$ has peak state transfer between $(u,x)$ and $(v,y)$ at time $\tau$. In this case, $|U_{X\square Y}(\tau)_{(u,x),(v,y)}|=B(X)_{u,v}\cdot B(Y)_{x,y}.$
\item If $x$ is periodic in $Y$ at time $\tau$, then $X\square Y$ admits peak state transfer between $(u,x)$ and $(v,x)$ at time $\tau$. In this case, $|U_{X\square Y}(\tau)_{(u,x),(v,x)}|=B(X)_{u,v}.$
\end{enumerate}
\end{theorem}

\begin{proof}
Suppose $X$ admits peak state transfer between $u$ and $v$ at time $\tau$. From (\ref{eq:cartprod}), we have 
\begin{center}
$|U_{X\square Y}(\tau)_{(u,x),(v,y)}|=|U_X(\tau)_{u,v}|\cdot |U_Y(\tau)_{x,y}|=B(X)_{u,v}\cdot |U_Y(\tau)_{x,y}|.$
\end{center}
So, if $Y$ admits peak state transfer between $x$ and $y$ at time $\tau$, the above equation and (\ref{eq:bounding}) yields 
\begin{center}
$B(X)_{u,v}\cdot B(Y)_{x,y}=|U_{X\square Y}(\tau)_{(u,x),(v,y)}|\leq B(X\square Y)_{(u,x),(v,y)}.$
\end{center}
Combining this with Proposition~\ref{lem:cartbound} yields $|U_{X\square Y}(\tau)_{(u,x),(v,y)}|=B(X)_{u,v}\cdot B(Y)_{x,y}=B(X\square Y)_{(u,x),(v,y)}.$ This proves 1, while 2 follows similarly.
\end{proof}

We note that Theorem~\ref{thm:peakstcartprod}(2) was first established by \'Arnad\'ottir and Guo \cite[Lemma 7.1]{arnadottir2026perfect}.

Let $X^{\square m}$ denote the Cartesian product of $m$ copies of $X$. If $X=Y$ in Theorem~\ref{thm:peakstcartprod}(1), then we obtain:

\begin{cor}
\label{cor:peakstcartprod}
If $X$ admits peak state transfer between $u$ and $v$ at time $\tau$, then $X^{\square m}$ has peak state transfer between $(u,\ldots,u)$ and $(v,\ldots,v)$ at time $\tau$ for all $m\geq 1$ and $|U_{X^{\square m}}(\tau)_{(u,\ldots,u),(v,\ldots,v)}|=\big(B(X)_{u,v}\big)^m.$
\end{cor}

Let $d\geq 1$ and $S$ be a set with $q\geq 2$ elements, The \textit{Hamming graph} $H(d,q)$ has vertex set equal to the set of $d$-tuples in $S$, and two vertices are adjacent if they differ in exactly one coordinate. Note that $H(d,q)=K_q^{\square d}$ and $H(d,2)$ is the hypercube. The next result follows from Corollary~\ref{cor:peakstcartprod} and Example~\ref{ex:knnn}.

\begin{cor}
For all $q,d\geq 2$, $H(d,q)$ has peak state transfer between $(u,\ldots,u)$ and $(v,\ldots,v)$ at time $\tau=\frac{\pi}{q}$. In this case, $|U_{H(d,q)}(\tau)_{(u,\ldots,u),(v,\ldots,v)}|=(\frac{2}{q})^d.$
\end{cor}

\begin{theorem}
\label{thm:pgstcartprod}
Suppose $X$ admits PGST between $u$ and $v$ relative to $\{\tau_m\}$.
\begin{enumerate}
\item There is PGST or PST between $x$ and $y$ in $Y$ relative to $\{\tau_m\}$ if and only if $X\square Y$ has PGST between $(u,x)$ and $(v,y)$ relative to $\{\tau_m\}$. 
\item A vertex $x$ is almost periodic or periodic in $Y$ relative to $\{\tau_m\}$ if and only if $X\square Y$ has PGST between $(u,x)$ and $(v,x)$ relative to $\{\tau_m\}$. 
\end{enumerate}
\end{theorem}

\begin{proof}
If $X$ has PGST between $u$ and $v$ in $X$ relative to $\{\tau_m\}$, then $\displaystyle\lim_{m\rightarrow\infty}$$|U_X(\tau_m)_{u,v}|=1$. Applying (\ref{eq:cartprod}), 
\begin{center}
$\displaystyle\lim_{m\rightarrow\infty}$$|U_{X\square Y}(\tau_m)_{(u,x),(v,y)}|=$ $\displaystyle\lim_{m\rightarrow\infty}$$|U_X(\tau_m)_{u,v}|\cdot $ $\displaystyle\lim_{m\rightarrow\infty}$$|U_Y(\tau_m)_{x,y}|=$$\displaystyle\lim_{m\rightarrow\infty}$$|U_Y(\tau_m)_{x,y}|$.
\end{center}
From this, 1 follows immediately, while 2 is straightforward from taking $x=y$.
\end{proof}

The following is immediate from Theorem~\ref{thm:pgstcartprod}.

\begin{ex}
Let $M\in\{A,L\}$ and suppose $X$ admits PGST between $u$ and $v$ relative to $\{\tau_m\}$.
\begin{enumerate}
\item Suppose $\{\tau_m\}\subseteq 2\pi\Z$. If $Y$ is an integral graph, then $X\square Y$ admits PGST between $(u,x)$ and $(v,y)$ relative to $\{\tau_m\}$ for all $x,y\in V(Y)$ relative to $\{\tau_m\}$. For example, if $M=A$, $P_n$ is a path with $n\in \{2^t,2p\}$, and $Y$ is an integral graph, then the preceding observation combined with Example~\ref{ex:path1} yields PGST in $P_n\square Y$ between $(u,x)$ and $(n+1-u,x)$ relative to $\{\tau_m\}$ for all $x,y\in V(Y)$.

\item Suppose $\{\tau_m\}\subseteq (2\Z+1)\frac{\pi}{2}$. If $Y$ admits PST between $u$ and $v$ at time $\frac{\pi}{2}$, then $X\square Y$ admits PGST between $(u,x)$ and $(v,x)$ relative to $\{\tau_m\}$. For example, if $X\in \{\operatorname{Cay}(\Z_n,S\cup\{\frac{n}{2}\}),\operatorname{Cay}(\Z_n,S\cup\{\frac{n}{4},\frac{3n}{4}\})\}$ as in Example~\ref{ex:Cn1}, and $Y=H(d,2)$, then our preceding observation yields PGST in $X\square Y$ between $(u,x)$ and $(u+\frac{n}{2},y)$ relative to $\{\tau_m\}$ for any pair of antipodal vertices $x$ and $y$ in $Y$.
\end{enumerate}
\end{ex}

\begin{theorem}
\label{thm:pgstpal}
Suppose that PST occurs between $u$ and $v$ in $X$ at time $\tau$ and between $x$ and $y$ in $Y$ at time $\tau'$ such that $\{1,\frac{\tau'}{\tau}\}$ is a linearly independent set over $\Q$. Then there exists a sequence $\{m_r\}$ of odd integers such that $\lim_{r\rightarrow\infty}|U_Y(m_r\tau)_{x,y}|=1$. Consequently, PGST occurs between $(u,x)$ and $(v,y)$ in $X\square Y$.
\end{theorem}

\begin{proof}
Let $\{1,\frac{\tau'}{\tau}\}$ be a linearly independent set over $\Q$. By Theorem~\ref{thm:kron}, for each $\epsilon>0$, there are integers $q,p$ such that $|q(\frac{\tau'}{\tau})-p-\frac{1}{2}(1-\frac{\tau'}{\tau})|<\frac{\epsilon}{2\tau}\pmod{2\pi}$. Equivalently, $|(2q+1) \tau'-(2p+1)\tau|<\epsilon \pmod{2\pi}$.
Thus, for each $\epsilon>0$, there exist sequences of odd integers $\{2q_r+1\}$ and $\{2p_r+1\}$ satisfying the property $\lim_{r\rightarrow\infty}|(2q_r+1) \tau'-(2p_r+1)\tau|=0$. Now, since PST occurs between $x$ and $y$ in $Y$ at every odd multiple of $\tau'$, we get
$|U_Y((2q_r+1)\tau')_{x,y}|=1$ for each $r$, and because $|U_Y(t)_{x,y}|$ is continuous, we get $\lim_{r\rightarrow\infty}|U_Y((2p_r+1)\tau)_{x,y}|=1$. Combining this with (\ref{eq:cartprod}) yields the desired conclusion.
\end{proof}

Theorem~\ref{thm:pgstpal} complements \cite[Theorem 4.2]{PAL2017746}, which states that if PST occurs between $u$ and $v$ in $X$ at time $\tau$ and $x$ is periodic in $Y$ at time $\tau'$ such that $\{1,\frac{\tau'}{\tau}\}$ is a linearly independent set over $\Q$, then PGST occurs between $(u,x)$ and $(v,x)$ in $X\square Y$. %One recovers this result by tweaking the proof of Theorem~\ref{thm:pgstpal}.

\begin{ex}
For all integers $m\geq 1$, $P_3^{\square m}$ admits PST at time $\frac{\pi}{\sqrt{2}}$. If $X$ is a regular graph admitting PST at time $\tau'$, then $\tau'=\frac{\pi}{g}$ for some integer $g\geq 1$. Applying Theorem~\ref{thm:pgstpal}, we obtain PGST in 
$(P_3^{\square m})\square X$. In particular, if $X=H(d,2)$, then $\tau'=\frac{\pi}{2}$, and so
$(P_3^{\square m})\square H(d,2)$ has PGST for all integers $m,d\geq 1$.
\end{ex}

\section{Graph complements}
\label{sec:comp}

The \textit{complement} of a graph $X$, denoted $X^c$, is the graph with the same vertex set as $X$ such that 
\begin{equation}
\label{eq:evalcomp}
A(X)+A(X^c)=\mathbf{J}-I\quad \text{and}\quad L(X)+L(X^c)=nI-\mathbf{J}.
\end{equation}
Note that $X$ is $k$-regular if and only if $X^c$ is $(n-k-1)$-regular. Since at least one of $X$ and $X^c$ is connected, we say that $X$ is \textit{coconnected} if $X$ and $X^c$ are both connected. Observe that if $X$ is connected but not coconnected, then $X$ is a join graph. 

Our goal in this section is to provide necessary conditions such that peak state transfer and PGST are preserved under graph complementation.
 
%The following basic result will be useful later.

\begin{prop}
\label{prop:boundmat}
Let $X$ be a coconnected graph. If (i) $M=A$ and $X$ is regular, or (ii) $M=L$, then the bounding matrices of $X$ and $X^c$ are equal.
\end{prop}

\begin{proof}
Since $X$ is coconnected, our assumptions in (i) and (ii) imply that $X$ and $X^c$ have the same set of spectral idempotents. From this, the result is immediate.
\end{proof}

\begin{prop}
\label{prop:XXbar}
Let $X$ be a graph on $n$ vertices. For all $t\in\R$, we have  
\begin{equation*}
U_{X^c}(t)
=
  \begin{cases}
    e^{-it}U_{X}(-t)+\left(\frac{e^{-it(k+1)}(e^{itn}-1)}{n}\right)\mathbf{J} & \text{if $M=A$ and $X$ is regular},\\
    e^{itn}U_{X}(-t)+\left(\frac{1-e^{itn}}{n}\right)\mathbf{J} & \text{if $M=L$}.
  \end{cases}
\end{equation*}
\end{prop}
\begin{proof}
Since $L(X^c)=nI-\mathbf{J}-L(X)$ and the matrices $nI$, $\mathbf{J}$, and $L(X)$ pairwise commute, we obtain $U_{X^c}(t)=e^{itn}e^{-it\mathbf{J}}e^{-itL(X)}=e^{itn}e^{-it\mathbf{J}}U_{X}(-t).$ Combining this with $e^{-it\mathbf{J}}=I+\left(\frac{e^{-itn}-1}{n}\right)\mathbf{J}$ gives us
\begin{center}
$U_{X^c}(t)=\left(e^{itn}I+\left(\frac{1-e^{itn}}{n}\right)\mathbf{J}\right)U_{X}(-t).$
\end{center}
Finally, since $\frac{1}{n}\mathbf{J}$ is a spectral idempotent for $L$, we get $\mathbf{J}U_{X}(-t)=U_{X}(-t)\mathbf{J}=\mathbf{J}$. Hence, the conclusion for $M=L$ is immediate from the above equation. The case $M=A$ and $X$ is regular follows similarly.
\end{proof}

The following is straightforward from Proposition~\ref{prop:XXbar}.

\begin{cor}
\label{cor:comppgstpeak}
Let $u$ and $v$ be two vertices in $X$ such that $B(X)_{u,v}=B(X^c)_{u,v}$.
If (i) $M=A$ and $X$ is regular or (ii) $M=L$, then the following hold.
\begin{enumerate}
\item If $X$ admits PGST between $u$ and $v$ relative to a sequence $\{\tau_m\}$, then $X^c$ admits PGST between $u$ and $v$ relative to the same sequence if and only if $\lim_{m\rightarrow\infty}n\tau_m\equiv 0\pmod{2\pi}$. 
\item If $X$ admits peak state transfer between $u$ and $v$ at time $\tau$, then $X^c$ also admits peak state transfer between $u$ and $v$ at the same time if and only if $n\tau\equiv 0\pmod{2\pi}$, in which case $U_{X^c}(t)_{u,v}=e^{-it}\overline{U_{X}(t)_{u,v}}$ and $|U_{X^c}(t)_{u,v}|=|U_{X}(t)_{u,v}|=B(X)_{u,v}$. In particular, if $X$ admits PST between $u$ and $v$ at time $\tau$ and $n\tau\equiv 0\pmod{2\pi}$, then $X^c$ admits PST between $u$ and $v$ at time $\tau$.
\end{enumerate}
\end{cor}

If $X$ is coconnected, then Proposition~\ref{prop:boundmat} implies that $B(X)_{u,v}=B(X^c)_{u,v}$ for all vertices $u,v$ in $X$. Thus, Corollary~\ref{cor:comppgstpeak} applies to coconnected graphs as long as $M=A$ and $X$ is regular, or $M=L$.

\begin{ex}
\label{ex:2expgst}
Let $X$ be a regular coconnected graph that admits PGST between $u$ and $v$ relative to a sequence $\{\tau_m\}\subseteq 2\pi\Z$. By Corollary~\ref{cor:comppgstpeak}(1), $X^c$ also has PGST between $u$ and $v$ relative to $\{\tau_m\}$. 
\end{ex}

\begin{ex}
\label{ex:2ex}
The following statements hold by Corollary~\ref{cor:comppgstpeak}(2).
\begin{itemize}
\item Let $n\equiv 0\pmod{4}$ so that $X=K_2\square K_n$ is a coconnected regular graph with PST at time $\frac{\pi}{2}$. Then $X^c$ is the direct product of $K_2$ and $K_n$, which has PST at $\frac{\pi}{2}$ between the same pairs of vertices.
\item Let $X$ be the Petersen graph, which is a coconnected 3-regular graph on 10 vertices and admits peak state transfer at time $\pi$. Then $X^c$ is the Johnson graph $J(5,2)$ which has peak state transfer at $\pi$.
\end{itemize}
\end{ex}
%love the pun
We complement Example~\ref{ex:knnn} by showing that any pair of nonadjacent vertices in a balanced complete multipartite graph with $k$ parts admits peak state transfer when $k$ is even. 

\begin{ex}
\label{ex:knntw}
Let $m\geq 2$ and $X$ be the balanced complete multipartite graph $K_{m,\ldots,m}$ with $k\geq 2$ parts, where $k$ is even. Then $X=(\bigcup_k K_m)^c$. Let $u$ and $v$ be nonadjacent vertices in $X$. In $X^c$, $u$ and $v$ belong to a connected component isomorphic to $K_m$, and so they admit peak state transfer at time $\tau=\frac{\pi}{m}$ by Example~\ref{ex:knnn}. Since $n=mk$ and $n\tau\equiv 0\pmod{2\pi}$ because $k$ is even, Corollary~\ref{cor:comppgstpeak}(2) implies that $X=(X^c)^c$ admits peak state transfer between any two nonadjacent vertices at time $\tau$.
\end{ex}

\section{Comparing $|U_X(t)_{u,v}|$ and $|U_{X^c}(t)_{u,v}|$}
\label{sec:compare}

Let $M\in\{A,L\}$. Set $\delta=-1$ whenever $M=A$ and $\delta=n$ otherwise. Given vertices $u$ and $v$ in a graph $X$, we define two nonnegative functions $f_{u,v}:\R\rightarrow\R$ and $g_{u,v}:\R\rightarrow\R$ as follows:
\begin{equation*}
f_{u,v}(t):=\big|U_{X^c}(t)_{u,v}-e^{it\delta}U_X(-t)_{u,v}\big|\quad \text{and} \quad g_{u,v}(t):=\big||U_{X^c}(t)_{u,v}|-|U_X(t)_{u,v}|\big|. 
\end{equation*}

The main goal of this section is to provide tight bounds for $f_{u,v}(t)$ and $g_{u,v}(t)$, which in turn relate $U_X(t)$ and $U_{X^c}(t)$ entrywise. These two expressions quantify the difference between the behavior of quantum state transfer between two vertices in a graph and its complement. 

First, we derive crude bounds for $g_{u,v}(t)$ and $f_{u,v}(t)$ using the bounding matrices of $X$ and $X^c$.

\begin{prop}
\label{prop:regupperbound2}
For all $u,v\in V(X)$ and for all $t\in\R$, we have 
\begin{center}
$g_{u,v}(t)\leq b\leq 1\quad $ and $\quad f_{u,v}(t)\leq B(X)_{u,v}+B(X^c)_{u,v}
\leq 2$,
\end{center}
where $b=\max\{B(X)_{u,v}, B(X^c)_{u,v}\}$. In particular, if $b=B(X)_{u,v}$, then $g_{u,v}(\tau)=b$ if and only if $X$ admits peak state transfer between $u$ and $v$ at time $\tau$ and $U_{X^c}(\tau)_{u,v}=0$.
\end{prop}

\begin{proof}
If $|U_{X^c}(t)_{u,v}|\geq |U_X(t)_{u,v}|$, then $g_{u,v}(t)=|U_{X^c}(t)_{u,v}|-|U_X(t)_{u,v}|\leq B(X^c)_{u,v}-|U_X(t)_{u,v}|\leq B(X^c)_{u,v}$. Otherwise, $g_{u,v}(t)\leq B(X)_{u,v}$. Thus, $g_{u,v}(t)\leq b$ for all $t\in\R$. The bounds for $f_{u,v}(t)$ follow from (\ref{eq:bounding}) and the triangle inequality.
\end{proof}

\begin{prop}
\label{prop:regupperbound1}
For all $u,v\in V(X)$, we have 
$g_{u,v}(t)\leq f_{u,v}(t)$ for all $t\in\R$.
\end{prop}

\begin{proof}
Since $U_X(-t)=\overline{U_X(t)}$, we get $U_X(-t)_{u,v}=\overline{U_X(t)_{u,v}}$, and so $|U_X(-t)_{u,v}|=|U_X(t)_{u,v}|$. By the triangle inequality, $g_{u,v}(t)=||U_{X^c}(t)_{u,v}|-|U_X(-t)_{u,v}||\leq |U_{X^c}(t)_{u,v}-e^{it\delta}U_X(-t)_{u,v}|=f_{u,v}(t).$
\end{proof}

Since the graphs in Example~\ref{ex:2ex} satisfy $n\tau\equiv 0\pmod{2\pi}$, Corollary~\ref{cor:comppgstpeak}(2) implies that $f_{u,v}(\tau)=g_{u,v}(\tau)=0$. Now, if we take $X=K_2\vee Y$ that admits PST between the vertices in $K_2$ at time $\tau$ relative to $A$ (see \cite[Theorem 13]{Kirkland2023}
for a characterization of PST in such graphs), then we have $U_{X^c}(\tau)_{u,v}=0$, and so $g_{u,v}(\tau)=1.$  We now give families where $f_{u,u}(\tau)=2$ and $\displaystyle\lim_{n\rightarrow\infty}f_{u,v}(\tau')=1$ for some $\tau,\tau'\in\R$.

\begin{ex}
\label{ex:k2n}
Let $M=A$. The following hold.
\begin{enumerate}
\item Let $n$ be an even perfect square and $u$ be the central vertex of $X=K_{1,r}$. Then $u$ is isolated in $X^c$, and so $U_{X^c}(t)_{u,u}=1$ for all $t$. Since $U_{X}(t)_{u,u}=\cos(\sqrt{r}t)$, $f_{u,u}\big(\pi\big)=
\big|1-e^{-i\pi}\cos(\sqrt{r}\pi)\big|=2.$
So, $f_{u,u}(\pi)=2$ for all even perfect square $n\geq 4$.
\item Let $u,v$ be vertices of degree $r$ in $X=K_{2,r}$. One checks that $U_X(\frac{\pi}{\sqrt{2r}})_{u,v}=-1$ and $U_{X^c}(\frac{\pi}{\sqrt{2}})_{u,v}=i\sin(\frac{\pi}{\sqrt{2r}})$. Hence, $f_{u,v}\big(\frac{\pi}{\sqrt{2r}}\big)=
\big|i\sin(\frac{\pi}{\sqrt{2r}})+e^{-i\frac{\pi}{\sqrt{2r}}}\big|=\big|\cos(\frac{\pi}{\sqrt{2r}})\big|.$
So, 
$\lim_{r\rightarrow\infty}f_{u,v}(\frac{\pi}{\sqrt{2r}})=1$.
\end{enumerate}
\end{ex}

%\subsection{Improved bounds for $g_{u,v}(t)$ and $f_{u,v}(t)$}

\begin{theorem}
\label{thm:regupperbound}
Suppose (i) $M=A$ and $X$ is regular or (ii) $M=L$. If $X$ has $n$ vertices, then for all $u,v\in V(X)$, $g_{u,v}(t)\leq f_{u,v}(t)=\frac{1}{n}|1-e^{itn}|\leq \frac{2}{n}$ for all $t\in\R$.
\end{theorem}

\begin{proof}
By Proposition~\ref{prop:XXbar}, we get that $f_{u,v}(t)=|U_{X^c}(t)_{u,v}-e^{it\delta}U_X(-t)_{u,v}|=\frac{1}{n}|1-e^{itn}|\leq 2/n$. Combining this with Proposition~\ref{prop:regupperbound1} yields the desired conclusion.
\end{proof}

The case $M=A$ and $X$ is regular was first established in \cite[Lemma 2.1]{SedQW}. Also, by virtue of Proposition~\ref{prop:Buv}, 
the upper bound of $\frac{2}{n}$ in Theorem~\ref{thm:regupperbound} improves that in Proposition~\ref{prop:regupperbound2}.
Thus, if (i) $M=A$ and $X$ is regular or (ii) $M=L$, then there are no graphs that behave similarly as those in Example~\ref{ex:k2n}.

As a consequence of Theorem~\ref{thm:regupperbound}, the behavior of quantum walks on $X$ and $X^c$ tends to coincide as the number of vertices increases.

\begin{cor}
\label{eq:cornlarge}
Let $X$ be a graph on $n$ vertices. If (i) $M=A$ and $X$ is regular or (ii) $M=L$, then for all $u,v\in V(X)$ and for all $t\in\R$, we have $f_{u,v}(t)\rightarrow 0$ and $g_{u,v}(t)\rightarrow 0$ as $n\rightarrow\infty$.
\end{cor}

The result in Corollary~\ref{eq:cornlarge} does not necessarily hold whenever $M=A$ and $X$ is irregular. Indeed, the graph in Example~\ref{ex:k2n}(2) satisfies $\lim_{r\rightarrow\infty}f_{u,v}(\frac{\pi}{\sqrt{2r}})=\lim_{r\rightarrow\infty}g_{u,v}(\frac{\pi}{\sqrt{2r}})=1$. Theorem~\ref{thm:regupperbound} also gives us:

\begin{cor}
\label{eq:corfg}
Let $X$ be graph on $n$ vertices. If (i) $M=A$ and $X$ is regular or (ii) $M=L$, then:
\begin{enumerate}
\item If $\tau=2j\pi/n$ for any integer $j$, then $|U_{X^c}(\tau)_{u,v}|=|U_X(\tau)_{u,v}|$ for all $u,v\in V(X)$.
\item %If vertex $u$ is periodic in $X$ at time $\tau$, then $|U_{X^c}(\tau)_{u,u}|\geq 1-\frac{2}{n}$. 
If PGST occurs between $u$ and $v$ in $X$ relative to $\{\tau_m\}$, then $\liminf_{m\rightarrow\infty}|U_{X^c}(\tau_m)_{u,v}|\geq 1-\frac{2}{n}$. If peak state transfer occurs between $u$ and $v$ in $X$ at time $\tau$, then $|U_{X^c}(\tau)_{u,v}|\geq B(X)_{u,v}-\frac{2}{n}$.
\end{enumerate}
\end{cor}

From Corollary~\ref{eq:cornlarge}, $|U_{X^c}(t)_{u,v}|$ and $|U_X(t)_{u,v}|$ tend to coincide as $n\rightarrow\infty$. This is complemented by Corollary~\ref{eq:corfg}(1), which states that $|U_{X^c}(t)_{u,v}|$ and $|U_X(t)_{u,v}|$ coincide for countably infinite values of $t$. %Moreover, Corollary~\ref{eq:corfg}(2) says that if $n$ is sufficiently large, then accurate transfer in $X$ translates to high probability transfer in $X^c$ between the same pairs of vertices.

\begin{cor}
\label{cor:joinfuvupperbound}
Let $X$ be a join graph on $n$ vertices with $u,v\in V(X)$. 
\begin{enumerate}
\item If $u$ and $v$ are in different components in $X^c$, then $g_{u,v}(t)=f_{u,v}(t)=|U_X(t)_{u,v}|\leq \frac{2}{n}$ for all $t\in\R$.

\item If $M=L$ and $u$ is an isolated vertex in $X^c$, then $|U_X(t)_{u,u}|\geq 1-\frac{2}{n}$ for all $t\in\R$. In particular, $u$ is sedentary in $X$ whenever $n\geq 3$.
\end{enumerate}
\end{cor}

\begin{proof}
The assumption of 1 yields $U_{X^c}(t)_{u,v}=0$ for all $t\in\R$, and so $f_{u,v}(t)=|U_X(-t)_{u,v}|g_{u,v}(t)$, so Theorem~\ref{thm:peakstjoin1} gives us 1. To prove 2, let $M=L$. If $u$ is isolated in $X^c$, then $|U_{X^c}(t)_{u,u}|=1$ for all $t\in\R$, so Theorem~\ref{thm:regupperbound} implies $g_{u,u}(t)=\big|1-|U_X(t)_{u,u}|\big|\leq \frac{2}{n}$. From this, 2 is immediate.
\end{proof}

\section{Non-regular graphs}
\label{sec:nonreg}

Throughout, we let $M=A$. Motivated by Theorem~\ref{thm:regupperbound}, we derive upper bounds for $f_{u,v}(t)$ whenever $X$ is irregular. In order to do this, we recall the concept of main eigenvectors.

An eigenvector $\mathbf{v}$ of $A(X)$ is \textit{main} if it is not orthogonal to $\mathbf{1}$. Otherwise, an eigenvector $\mathbf{v}$ is called \textit{non-main}. An eigenvalue $\lambda\in\spec(X)$  is \textit{main} if $E_\lambda\one\neq 0$ (i.e., at least one of its associated eigenvectors is main). Otherwise, an eigenvalue $\lambda\in \spec(X)$ is called \textit{non-main}. We denote the set of main and non-main eigenvalues of $X$ by $\mspec(X)$ and $\nspec(X)$, respectively.

We state two basic results about main eigenvectors that are useful in this paper. The first follows from the proof of \cite[Proposition 1.3]{akbari2026switching}, and the second is immediate from (\ref{eq:evalcomp}).

\begin{prop}
\label{prop:maineigsp}
If $\lambda\in\mspec(X)$, then there exists an orthonormal basis for the eigenspace of $\lambda$ consisting of all main eigenvectors.
\end{prop}

\begin{prop}
\label{prop:main}
We have that $\bv$ is a non-main eigenvector associated with $\lambda\in\spec(X)$ if and only if $\bv$ is non-main eigenvector associated with $-(1+\lambda)\in\spec(X^c)$.
\end{prop}

We show that
the matrix $U_{X^c}(t)-e^{-it}U_X(-t)$ is determined by the main eigenvectors of $X$ and $X^c$.

\begin{lemma}
\label{lem: transition_difference}
Let $\{\mathbf{v}_1,\dots, \mathbf{v}_m,\mathbf{v}_{m+1}\dots,\mathbf{v}_n\}$ and  $\{\mathbf{u}_1,\dots,\mathbf{u}_m,\mathbf{v}_{m+1}\dots,\mathbf{v}_n\}$ be orthonormal bases of eigenvectors of $A(X)$ and $A(X^c)$ respectively, where $\bu_j, \bv_j$ are main for all $j\in\{1,\ldots,m\}$ and non-main otherwise. Then
    \begin{equation*}
    U_{X^c}(t)-e^{-it}U_X(-t) = \sum_{j = 1}^m\big(e^{it\lambda_j}\mathbf{u}_j\mathbf{u}_j^\top - e^{-it(1+\eta_j)}\mathbf{v}_j\mathbf{v}_j^\top\big)\quad \text{for all $t\in\R$},
\end{equation*}
where $\lambda_j$ and $\eta_j$ are the main eigenvalues of $A(X^c)$ and $A(X)$ corresponding to $\mathbf{u}_j$ and $\mathbf{v}_j$, respectively. 
\end{lemma}
\begin{proof}
     Observe that
\begin{equation*}
    U_X(t) = \sum_{j = 1}^me^{it\eta_j}\mathbf{v}_j\mathbf{v}_j^\top + \sum_{j=m+1}^ne^{it(\eta_j)}\mathbf{v}_j\mathbf{v}_j^\top.
\end{equation*}
By Proposition~\ref{prop:main}, $\mathbf{v}_j$ is a non-main eigenvector of $A(X)$ with eigenvalue $\eta_j$ if and only if $\bv_j$ is a non-main eigenvector of $A(X^c)$ with eigenvalue $-(1+\eta_j)$. Consequently, 
\begin{equation*}
    U_{X^c}(t)= \sum_{i = j}^me^{it\lambda_j}\mathbf{u}_j\mathbf{u}_j^\top + \sum_{j=m+1}^ne^{-it(1+\eta_j)}\mathbf{v}_j\mathbf{v}_j^\top
\end{equation*}
The claim then follows.
\end{proof}

As spectral idempotents sum to $I$, $\bv_j$'s and $\bu_j$'s in Lemma~\ref{lem: transition_difference} satisfy $\sum_{j=1}^m\bv_j\bv_j^\top=\sum_{j=1}^m\bu_j\bu_j^\top$. Let
\begin{center}
$
P:=\displaystyle\sum_{j=1}^m\bv_j\bv_j^\top=\sum_{j=1}^m\bu_j\bu_j^\top.$
\end{center}
The \textit{orthonormalized walk matrix} of a graph $X$ is the matrix $W$ whose columns are the orthonormalized elements of $\{\one,A\one,A^2\one,\dots, A^{m-1}\one\}$. From \cite[Theorem 2.4]{rowlinson2007maineigenvaluesurvery}, a basis for the image of $P$ is $\{\one,A\one,A^2\one,\dots, A^{m-1}\one\}$, meaning $P=WW^T$. Using this construction of $P$, we obtain the following bound on $f_{u,v}(t)$ that depends only on $A$.

\begin{theorem}
    \label{thm:projector upper bound}
If $u,v$ are vertices in a graph $X$, then $f_{u,v}(t) \leq 2\sqrt{P_{u,u}}\sqrt{P_{v,v}}.$
\end{theorem}

\begin{proof}
    Applying Lemma~\ref{lem: transition_difference} and the triangle inequality, we have
    \[f_{u,v}(t) \leq \sum_{j = 1}^m\left| (\mathbf{u}_j\mathbf{u}_j^\top)_{u,v}\right|+ \sum_{j=1}^m \left|(\mathbf{v}_j\mathbf{v}_j^\top)_{u,v}\right|.\]
    Of course $\left|(\mathbf{u}_j\mathbf{u}_j^\top)_{u,v}\right|=\left| (\mathbf{u}_j^\top\be_u)(\mathbf{u}_j^\top\be_v)\right|$, and so the Cauchy-Schwarz inequality yields
    \[f_{u,v}(t)\leq \sqrt{\sum_{j=1}^m(\mathbf{u}_j^\top\be_u)^2}\sqrt{\sum_{j=1}^m(\mathbf{u}_j^\top\be_v)^2} + \sqrt{\sum_{j=1}^m(\mathbf{v}_j^\top\be_u)^2}\sqrt{\sum_{j=1}^m(\mathbf{v}_j^\top\be_v)^2}.\]
    Since $P_{u,u}=\sum_{j=1}^m(\mathbf{u}_j^\top\be_u)^2=\sum_{j=1}^m(\mathbf{v}_j^\top\be_u)^2$, we obtain \(f_{u,v}(t) \leq 2\sqrt{P_{u,u}}\sqrt{P_{v,v}}\) as claimed.
\end{proof}

\begin{rem}
    The number of main eigenvalues of a graph is known to be the largest integer $m$ such that the set $\{\one, A\one,A^2\one, \dots, A^{m-1}\one\}$ is linearly independent \cite[Proposition 2.2]{rowlinson2007maineigenvaluesurvery}. Hence, Theorem~\ref{thm:projector upper bound} can be applied to any graph $X$ without knowing the spectral decomposition of $A(X)$ or $A(X^c)$. 
\end{rem}
  
We examine special cases of Theorem~\ref{thm:projector upper bound}. A graph $X$ is regular if and only if $\one$ is the only main eigenvector of $A(X)$ \cite[Proposition 1.4]{rowlinson2007maineigenvaluesurvery}. So  we may view Proposition~\ref{prop:XXbar} with $M=A$ and $X$ being regular as a special case of Theorem~\ref{thm:projector upper bound}. A connected graph $X$ is \textit{controllable} if all eigenvalues of $A(X)$ are simple and main \cite{Cvetokvic2010ControlGraphs}. For such a graph, $P=I$, and so Theorem~\ref{thm:projector upper bound} yields $f_{u,v}(t)\leq 2$, which matches Proposition~\ref{prop:regupperbound2}. For large values of $n$, almost all graphs are controllable \cite{ORourkeTouri2016ControllableGraphs}.

For connected graphs with two linearly independent  main eigenvectors, we observe the following.

\begin{prop}
\label{prop:2mainevec}
Let $X$ be a connected graph. Then $A(X)$ has exactly two linearly independent main eigenvectors if and only if $A(X)$ has exactly two main eigenvalues that are both simple. In this case, one of the main eigenvectors is a Perron eigenvector.
\end{prop}

\begin{proof}
Let $\bu$ and $\bv$ be the two linearly independent main eigenvectors of $A(X)$. Since a Perron eigenvector has all positive entries, it follows that it is a main eigenvector, and hence it must be one of $\bu$ or $\bv$. Suppose $\bu$ is a Perron eigenvector and $\bv$ is a main eigenvector for some eigenvalue $\lambda$. If $\lambda$ has multiplicity at least two, then Proposition~\ref{prop:maineigsp} and the fact that the eigenvectors of a symmetric matrix corresponding to distinct eigenvalues are orthogonal imply that there are at least three linearly independent main eigenvectors, a contradiction. Thus, $\lambda$ is a simple eigenvalue of $A(X)$. The converse is immediate.
\end{proof}

As a demonstration of Theorem~\ref{thm:projector upper bound}, we now derive an upper bound for $f_{u,v}(t)$ when $X$ is irregular and $A(X)$ has two linearly independent main eigenvectors. Let $d_u$ and $d_{avg}$ denote the degree of $u$ in $X$ and the average degree of $X$, respectively.

\begin{theorem}
\label{thm:nonreguv}
Let $X$ be a graph on $n$ vertices with average degree $d_{avg}$. If $A(X)$ has exactly two linearly independent main eigenvectors, then for all $u,v\in V(X)$,
    \[f_{u,v}(t) \leq \frac{2}{n}\sqrt{\bigg(1+\frac{n(d_u-d_{avg})^2}{\sum_{w\in V(X)}(d_w-d_{avg})^2}\bigg)\bigg(1+\frac{n(d_v-d_{avg})^2}{\sum_{w\in V(X)}(d_w-d_{avg})^2}\bigg)},\quad \text{for all $t\in\R$}.\]
\end{theorem}

\begin{proof}
Recall that the image of $P$ has basis $\{\one,A\one\}$. The set $\left\{\frac{\one}{\sqrt{n}},\frac{A\one - d_{avg}\one}{\sqrt{\sum_{w\in V(X)}(d_w-d_{avg})^2}}\right\}$ is an orthonormalization of $\{\one,A\one\}$, which we use to compute $P$. We find that
    \[P_{u,u} = \frac{1}{n} + \frac{(d_u-d_{avg})^2}{\sum_{w\in V(X)}(d_w-d_{avg})^2}.\]
The claim then follows from this fact and Theorem~\ref{thm:projector upper bound}.
\end{proof}

\begin{rem}
In Theorem~\ref{thm:nonreguv}, if $u,v$ both have average degree, then the upper bound for $f_{u,v}(t)$ coincides with that in Theorem~\ref{thm:regupperbound}. However, if at least one of $u,v$ has degree not equal to the average degree, then Theorem~\ref{thm:nonreguv} yields a larger upper bound for $f_{u,v}(t)$ than that in Theorem~\ref{thm:regupperbound}. 
\end{rem}

If a graph $X$ has vertices with only two distinct degrees and $A(X)$ has exactly two linearly independent main eigenvectors, then Theorem~\ref{thm:nonreguv} cleans up nicely. 

\begin{cor}
    \label{cor:twomain bidegree bound}
    Let $X$ be a $n_1+n_2$ vertex graph with exactly two linearly independent main eigenvectors, $n_1$ vertices of degree $k$, and $n_2$ vertices of degree $\ell$. Then 
    \[f_{u,v}(t) \leq \begin{cases}
        2/n_1, &d_u=d_v=k\\
        2/n_2, &d_u=d_v=\ell\\
        2/\sqrt{n_1n_2}&d_u\neq d_v.\\
    \end{cases}\]
\end{cor}
\begin{proof}
Note that $d_{avg} = \frac{n_1k + n_2\ell}{n_1+n_2}$. If $d_u=k$, then $(d_u-d_{avg})^2 = \frac{n_2^2(k-\ell)^2}{(n_1+n_2)^2}$. If $d_u=\ell$, then $(d_u-d_{avg})^2 = \frac{n_1^2(k-\ell)^2}{(n_1+n_2)^2}$. Observe that $\sum_{w\in V(X)}(d_w-d_{avg})^2 = \frac{n_1n_2(k-\ell)^2}{(n_1+n_2)}$. The expression $1+\frac{n(d_u-d_{avg})^2}{\sum_{w\in V(X)}(d_w-d_{avg})^2}$ then simplifies to $1+\frac{n_2}{n_1}$ if $d_u=k$, and $1+\frac{n_1}{n_2}$ if $d_u=\ell$. The claim follows from this and Theorem~\ref{thm:nonreguv}.  
\end{proof}

By Proposition~\ref{prop:2mainevec}, graphs that meet the hypotheses of Corollary~\ref{cor:twomain bidegree bound} include nonregular joins of regular graphs and a family of coconnected split graphs that are explored in Section~\ref{sec:split}. We now demonstrate that with careful analysis of $X$ and $X^c$, it is possible to acquire bounds tighter than what Theorem~\ref{thm:projector upper bound} provides.

\begin{ex}
\label{ex:projectorboundonjoin}
Let $X,Y$ be $k,\ell$-regular graphs on $n_1,n_2$ vertices, respectively. If $k+n_2=\ell+n_1$, then $X\lor Y$ is regular and $f_{u,v}(t)\leq \frac{2}{n_1+n_2}$ by Theorem~\ref{thm:regupperbound}. If $k+n_2\neq \ell+n_1$, then $X\lor Y$ is irregular with all vertices having degree $k+n_2$ or $\ell+n_1$, and $A(X\lor Y)$ has exactly two linearly independent main eigenvectors. Thus,  Corollary $\ref{cor:twomain bidegree bound}$ grants us bounds for $f_{u,v}(t)$. When $u,v$ both belong to $X$ or to $Y$, Corollary~\ref{cor:joinfuvupperbound} and Corollary~\ref{cor:twomain bidegree bound} yield the same upper bound. When $u\in V(X)$ and $v\in V(Y)$, $U_{(X\lor Y)^c}(t)_{u,v}  = 0$, so $f_{u,v}(t) = \big|U_{X\lor Y}(t)_{u,v}\big|$. By Theorem~\ref{thm:peakstjoin1}, $f_{u,v}(t) \leq 2/\sqrt{(k-\ell)^2+4n_1n_2}$. In this case Corollary~\ref{cor:twomain bidegree bound} yields the larger upper bound $f_{u,v}(t)\leq 2/\sqrt{n_1n_2}$.
\end{ex}

\section{Nordhaus-Gaddum type inequalities}
\label{sec:NG}

Inspired by the classical Nordhaus–Gaddum inequalities involving the sum and product of a graph parameter for a graph and its complement, we will provide crude upper bounds for the expressions
\begin{equation*}
\sigma_{u,v}(t):=|U_X(t)_{u,v}|+|U_{X^c}(t)_{u,v}|\quad \text{and} \quad \rho_{u,v}(t):=|U_X(t)_{u,v}|\cdot |U_{X^c}(t)_{u,v}|
\end{equation*}
where $u$ and $v$ are two vertices in a graph $X$. 

Note that a trivial lower bound for these two expressions is 0. In what follows, we show that this trivial lower bound cannot be improved for $\rho_{u,v}(t)$ whenever $u\neq v$ or $u=v$ is not sedentary.

\begin{prop}
If $u\neq v$, then $\displaystyle\inf_{t>0}\rho_{u,v}(t)=0$. Moreover, if $u$ is not sedentary in $X$, then $\displaystyle\inf_{t>0}\rho_{u,u}(t)=0$.
\end{prop}

\begin{proof}
Since each vertex is almost periodic in $X$, $\sup_{t>0}\lvert U_X(t)_{u,u}\rvert=1$. So, by the unitarity of $U_X(t)$, $\inf_{t>0}\lvert U_X(t)_{u,v}\rvert=0$ if $u\neq v$.
If  $u$ is not sedentary in $X$, then $\inf_{t>0}\lvert U_X(t)_{u,u}\rvert=0$. Combining these observations with the fact that for all $u,v\in V(X)$, $\rho_{u,v}(t)\leq U_X(t)_{u,v}$ for all $t$ yields the above result.
\end{proof}

\begin{prop}
\label{prop:NG}
Let $X$ be a graph with $u,v\in V(X)$. Then for all $t\in\mathbb{R}$, 
\begin{equation}
\label{eq:peakst1}
\sigma_{u,v}(t)\leq B(X)_{u,v}+B(X^c)_{u,v}\quad \text{and}\quad \rho_{u,v}(t)\leq B(X)_{u,v}\cdot B(X^c)_{u,v}.
\end{equation}
Moreover, the following holds.
\begin{enumerate}
\item Let $X$ be coconnected. The following are equivalent: (i) peak state transfer occurs between $u$ and $v$ in $X$ and $X^c$ at time $\tau$, (ii) the equality holds for $\sigma_{u,v}(\tau)$ in (\ref{eq:peakst1}), and (iii) the equality holds for $\rho_{u,v}(\tau)$ in (\ref{eq:peakst1}). The following are also equivalent: (a) $u$ is periodic in $X$ and $X^c$ at time $\tau$, (b) $\sigma_{u,u}(\tau)=2$, and (c) $\rho_{u,u}(\tau)=1$. Moreover, the following are equivalent: (A) PGST occurs between $u$ and $v$ in $X$ and $X^c$ relative to the same sequence of times, (B) $\sup_{t>0}\sigma_{u,v}(t)=2$ and (C) $\sup_{t>0}\rho_{u,v}(t)=1$. Furthermore, the following are also equivalent: (I) $u$ is almost periodic in $X$ and $X^c$ relative to the same sequence of times, (II) $\sup_{t>0}\sigma_{u,u}(t)=2$ and (III) $\sup_{t>0}\rho_{u,u}(t)=1$.
\item Let $X$ be coconnected. If (i) $M=A$ and $X$ is regular or (ii) $M=L$, then $\sigma_{u,v}(t)\leq 2B(X)_{u,v}$ and $\rho_{u,v}(t)\leq (B(X)_{u,v})^2$ for all $t\in\mathbb{R}$.
\item Suppose $X^c$ is a disconnected graph on $n$ vertices where $u,v$ belong to different components in $X^c$. Then $\sigma_{u,v}(t)\leq B(X)_{u,v}$ and $\rho_{u,v}(t)=0$ for all $t\in\mathbb{R}$. In particular, we have $\sigma_{u,v}(t)\leq \frac{2}{\sqrt{\Delta}}$ if $M=A$ and $X$ is a join of two regular graphs, and $\sigma_{u,v}(t)\leq \frac{2}{n}$ if  $M=L$.
\end{enumerate}
\end{prop}

\begin{proof}
Note that (\ref{eq:peakst1}) follows from (\ref{eq:bounding}). If $X$ is coconnected, then  $B(X)_{u,v},B(X^c)_{u,v}>0$ by Proposition~\ref{prop:B+}. Combining this with (\ref{eq:peakst1}) yields the first two statements in 1. The last two statements follow from the fact that $\sup_{t>0}\sigma_{u,v}(t)=2$ (resp., $\sup_{t>0}\rho_{u,v}(t)=1$ ) if and only if $\lim_{m\rightarrow\infty}|U_X(\tau_m)_{u,v}|=\lim_{m\rightarrow\infty}|U_{X^c}(\tau_m)_{u,v}|=1$ for some $\{\tau_m\}$. Next, combining (\ref{eq:peakst1}) and Proposition~\ref{prop:boundmat} gives us 2. Finally, the assumption in 2 implies that $B(X^c)_{u,v}=0$. So, combining this with (\ref{eq:peakst1}) and Lemma~\ref{lem:joinbound} yields 3.
\end{proof}

Note that if Proposition~\ref{prop:NG}(1i) holds and the peak state transfer is PST, then $\sigma_{u,w}(\tau)=\rho_{u,w}(\tau)=0$ for all $w\neq v$. Moreover, if Proposition~\ref{prop:NG}(1A) (resp., (1I)) holds, then $\inf_{t>0}\sigma_{u,w}(t)=\inf_{t>0}\rho_{u,w}(t)=0$ for all $w\neq v$ (resp., $w\neq u$). Next, we show that the bounds in (\ref{eq:peakst1}) in Proposition~\ref{prop:NG} are not always met.

\begin{ex}\label{ex:K2r}
Let $r=2s^2$ and $u,v$ be vertices of degree $r$ in $X=K_{2,r}$. Then $B(X)_{u,v}=B(X^c)_{u,v}=1$, $U_X(t)_{u,v}=\frac{-1+\cos(2st)}{2}$ and $U_{X^c}(t)_{u,v}=i\sin(t)$. So, $|U_X(t)_{u,v}|=\sin^2(st)$ and $|U_{X^c}(t)_{u,v}|=|\sin t|$ are periodic functions in $t$ with periods $\frac{\pi}{s}$ and $\pi$, respectively. If $s$ is odd, then $|U_X(\frac{\pi}{2})_{u,v}|=|U_{X^c}(\frac{\pi}{2})_{u,v}|=1$, and so $\sigma_{u,v}(\frac{\pi}{2})=2$ and $\rho_{u,v}(\frac{\pi}{2})=1$. If $s$ is even, then for any odd integer $q$, $|U_{X^c}(\frac{q\pi}{2})_{u,v}|=1$ and $|U_X(\frac{q\pi}{2})_{u,v}|=0$, while $|U_{X^c}(\frac{q\pi}{2s})_{u,v}|<1=|U_X(\frac{q\pi}{2s})_{u,v}|$. These calculations imply that $|U_X(t)_{u,v}|$ and $|U_{X^c}(t)_{u,v}|$ do not simultaneously attain 1, and so by Proposition~\ref{prop:NG}, the inequalities in (\ref{eq:peakst1}) are strict. Since the sum and product of $|U_X(t)_{u,v}|$ and $|U_{X^c}(t)_{u,v}|$ are periodic continuous functions, $\sigma_{u,v}(t)$ and $\rho_{u,v}(t)$ both attain their maximum, which are less than 2 and 1, respectively. For instance, if $s=2$, then $\max_{t>0}\sigma_{u,v}(t)\approx 1.736<2$ and $\max_{t>0}\rho_{u,v}(t)\approx 0.7436<1$.
\end{ex}

Combining Proposition~\ref{prop:NG}(1i) and Example~\ref{ex:2ex}, we obtain graphs that in fact attain the upper bound for $\sigma_{u,v}(t)$ and $\rho_{u,v}(t)$ in (\ref{eq:peakst1}). We also note that the graphs in Example~\ref{ex:2expgst} satisfy Proposition~\ref{prop:NG}(1B). We now complement Proposition~\ref{prop:NG}(1B) with the next result.

\begin{cor}
\label{cor:pgstsup}
If $X$ and $X^c$ both admit PST between the same pair of vertices $u$ and $v$ at times $\tau$ and $\tau'$ such that $\{1,\frac{\tau'}{\tau}\}$ is a linearly independent set over $\Q$, then $\sup_{t>0}\sigma_{u,v}(t)=2$ and $\sup_{t>0}\rho_{u,v}(t)=1$.
\end{cor}

\begin{proof}
Taking $Y=X^c$ and $(x,y)=(u,v)$ in Theorem~\ref{thm:pgstpal}, we get $\lim_{r\rightarrow\infty}|U_Y(m_r\tau)_{u,v}|=1$ for some sequence $\{m_r\}$ of odd integers. Since $|U_X(m_r\tau)_{u,v}|=1$, the result follows.
\end{proof}

\begin{ex}
\label{ex:p3}
Let $M=A$, and $X=P_3$ with degree one vertices $u$ and $v$. Then $X$ and $X^c$ admit PST between $u$ and $v$ at times $\tau=\frac{\pi}{\sqrt{2}}$ and $\tau'=\frac{\pi}{2}$, respectively. Since $\{1,\frac{\tau'}{\tau}\}$ is a linearly independent set over $\Q$, Corollary~\ref{cor:pgstsup} yields $\sup_{t>0}\sigma_{u,v}(t)=2$ and $\sup_{t>0}\rho_{u,v}(t)=1$.
\end{ex}

We now give bounds for $\sigma_{u,v}(t)$ and $\rho_{u,v}(t)$ whenever $u$ is sedentary.

\begin{prop}
\label{prop:NGsed}
Let $M\in\{A,L\}$ and $X$ be a graph with $u,v\in V(X)$ with $u\neq v$.
\begin{enumerate}
\item If $u$ is $C_1$-sedentary in $X$, then for all $t\in\mathbb{R}$,
\begin{equation}
\label{eq:NGsed}
C_1\leq \sigma_{u,u}(t)\quad \text{and}\quad \sigma_{u,v}(t)\leq 1+\sqrt{1-C_1^2}
\end{equation}
The left (resp., right) of (\ref{eq:NGsed}) is an equality if for some $\tau\in\R$, $|U_X(\tau)_{u,u}|=C_1$ and $|U_{X^c}(\tau)_{u,u}|=0$ (resp., $|U_X(\tau)_{u,u}|=\sqrt{1-C_1^2}$ and $|U_{X^c}(\tau)_{u,u}|=1$).
\item If $u$ is $C_1$-sedentary in $X$ and $C_2$-sedentary in $X^c$, then for all $t\in\mathbb{R}$, 
\begin{equation}
\label{eq:NGsed1}
C_1+C_2\leq \sigma_{u,u}(t)\quad \text{and}\quad \sigma_{u,v}(t)\leq \sqrt{1-C_1^2}+\sqrt{1-C_2^2}\quad 
\end{equation}
and
\begin{equation}
\label{eq:NGsed2}
C_1C_2\leq \rho_{u,u}(t)\quad \text{and}\quad \rho_{u,v}(t)\leq \sqrt{(1-C_1^2)(1-C_2^2)}.
\end{equation}
The left (resp., right) of (\ref{eq:NGsed1}) and (\ref{eq:NGsed2}) become equalities if for some $\tau\in\R$, $|U_X(\tau)_{u,u}|=C_1$ and $|U_{X^c}(\tau)_{u,u}|=C_2$ (resp., $|U_X(\tau)_{u,u}|=\sqrt{1-C_1^2}$ and $|U_{X^c}(\tau)_{u,u}|=\sqrt{1-C_2^2}$).
\end{enumerate}
\end{prop}

\begin{proof}
This is immediate from the definition of a $C$-sedentary vertex and Proposition~\ref{prop:sedpeak}.  
\end{proof}

\begin{ex}
Let $T\geq 3$ be a twin set in $X$, which is a maximal subset of $V(X)$ whose pairwise elements share the same neighbors. Note that $T$ is also a twin set in $X^c$. Invoking \cite[Theorem 16]{Monterde2023}, each vertex in $T$ is $(1-\frac{2}{|T|})$-sedentary in $X$ and $X^c$. Taking $C_1=C_2=1-\frac{2}{|T|}$ in Proposition~\ref{prop:NGsed}(2), we get
\begin{center}
$2(1-\frac{2}{|T|})\leq \sigma_{u,u}(t),\quad$ $\sigma_{u,v}(t)\leq \frac{4\sqrt{|T|-1}}{|T|},\quad$ $(1-\frac{2}{|T|})^2\leq \rho_{u,u}(t),\quad $ and $\quad \rho_{u,v}(t)\leq \frac{4(|T|-1)}{|T|^2}$ 
\end{center}
for all $u,v\in T$ and for all $t\in\R$. For example, each part in $X=K_{r,r}$ is a twin set with $|T|=r$. In this case, we get $\sigma_{u,v}(\frac{\pi}{n})=2(1-\frac{2}{r})$ and $\rho_{u,u}(\frac{\pi}{n})=(1-\frac{2}{r})^2$, and Example~\ref{ex:knntw} yields $\sigma_{u,v}(t)=\frac{2}{r}\leq \frac{4\sqrt{r-1}}{r}$ and $\rho_{u,v}(t)\leq \frac{1}{r^2}<\frac{4(r-1)}{r^2}$. Hence, the lower bounds in Proposition~\ref{prop:NGsed}(2) are attainable.
\end{ex}

We now show that $\sigma_{u,u}(t)$ can be bounded away from 0 without $u$ being sedentary in $X$ and $X^c$.

\begin{prop}
\label{prop:sigma}
Suppose that vertex $u$ is periodic in $X$ and $X^c$ at the same time. If $|U_X(t)_{u,u}|$ and $|U_{X^c}(t)_{u,u}|$ have no common zeros, then there exists $c>0$ such that $\sigma_{u,u}(t)\geq c$ for all $t\in\R$.
\end{prop}

\begin{proof}
Note that $\sigma_{u,u}(\tau)=0$ if and only if $|U_X(\tau)_{u,u}|=|U_{X^c}(\tau)_{u,u}|=0$. Thus, if $|U_X(t)_{u,u}|$ and $|U_{X^c}(t)_{u,u}|$ have no common zeros, then $\sigma_{u,u}(\tau)>0$ for all $t$. However, since $|U_X(\tau)_{u,u}|$ and $|U_{X^c}(\tau)_{u,u}|$ are both periodic functions with the same period, it follows that $\sigma_{u,u}(\tau)$ is also periodic. Since $\sigma_{u,u}(\tau)$ is continuous, these observations imply that $\sigma_{u,u}(\tau)$ attains a minimum value $c>0$.
\end{proof}

\begin{ex}
For $M=A$, let $r=2s^2$ where $s$ is even, and let $u,v$ be the vertices of degree $r$ in $X=K_{2,r}$, as in Example~\ref{ex:K2r}. One checks that $|U_X(t)_{u,u}|=\frac{1}{2}(1+\cos(2st))$ and $|U_{X^c}(t)_{u,u}|=|\cos(t)|$, which are periodic functions in $t$. Since the set of zeros of $|U_X(t)_{u,u}|$ and $|U_{X^c}(t)_{u,u}|$ are $\{\frac{p\pi}{2s}:p\ \text{is odd}\}$ and $\{\frac{q\pi}{2}:q\ \text{is odd}\}$, respectively. Since $s$ is even, $|U_X(t)_{u,u}|$ and $|U_{X^c}(t)_{u,u}|$ have no common zeros, and so by Proposition~\ref{prop:sigma}, we have $\sigma_{u,u}(t)>\beta$ for some $\beta>0$. 
%with periods $\frac{\pi}{s}$ and $\pi$, respectively. Since $u$ and $v$ are twins in $X$, we get $|U_G(t)_{u,u}| + |U_{G}(t)_{u,v}| \geq 1$ where $G\in\{X,X^c\}$ \cite[Equation 6]{kirklandMonterde2023quantum}. Consequently, $\sigma_{u,u} (t) + \sigma_{u,v}(t) \geq 2$. From Example~\ref{ex:K2r}, $\sigma_{u,v}(t)$ attains a maximum less than $2$, and so $\sigma_{u,u} (t) \geq 2 - \sigma_{u,v}(t) \geq 2 - \max_{t>0} \sigma_{u,v}(t) > 0$. Since $u$ is involved in PST in $X$ and $X^c$, we get that $u$ is not sedentary in both $X$ and $X^c$, and yet $\sigma_{u,u} (t)=\left|\frac{1}{2}\left(1+\cos\left(2st\right)\right)\right|+\left|\cos t\right|> 0$. 
For instance, if $s=2$, then $\min_{t>0}\sigma_{u,u} (t)\approx 0.672$.
\end{ex}

\section{Split graphs}
\label{sec:split}

A \emph{split graph} is a graph whose vertex set can be partitioned into $C\cup S$ where the subgraph induced by $C$ forms a clique while the subgraph induced by $S$ forms an independent set. %The family of  split graphs is closed under complementation, i.e., the complement of a split graph is still a split graph. %This makes split graphs of interest particularly with regards to quantum walks in graph complements. 
A split graph is \emph{bidegreed} if it has exactly two valencies, which we denote by $\alpha$ and $\beta$. If a split graph is bidegreed, then all vertices in $C$ have the same valency. Similarly, for all vertices in $S$. We define an $(\alpha,\beta)$\emph{-bidegreed} split graph as a bidegreed split graph whose vertices in $C$ have valency $\alpha$ and vertices in $S$ have valency $\beta$. Examples of bidegreed split graphs include the family of stars, double stars, the corona of a complete graph with an empty graph, and the join of a complete graph and an empty graph, see Fig.~\ref{fig2}.
%\linebreak\vspace{-10pt}

In this section, we compute the spectral decomposition of bidegreed split graphs and identify subfamilies that admit sedentariness in the graph and its complement. These families are interesting because they are nonregular coconnected graphs that in general do not contain twins, and are in general neither joins nor Cartesian products, unlike most known examples of graphs with sedentary vertices in the literature.

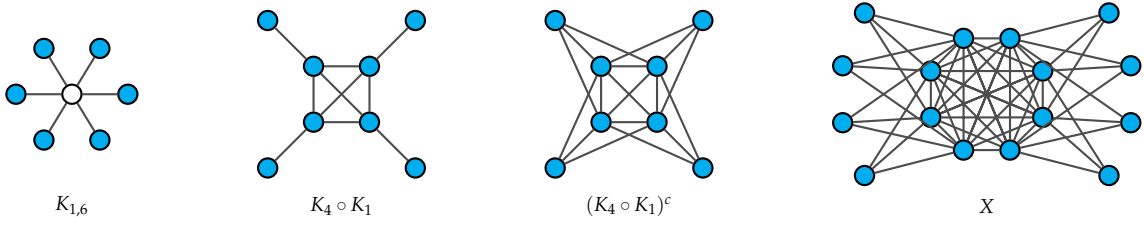
\begin{figure}
	\begin{center}
		\begin{tikzpicture}[scale=.37,auto=left]
		\tikzstyle{every node}=[circle, thick, fill=white, scale=0.7]
		
		%label={above, yshift=0cm: {\Large$u$}}
		
		\node[draw,fill=cyan] (1) at (-1,1.64) {};
		\node[draw,fill=cyan] (2) at (1,1.64)   {};
		\node[draw,fill=cyan] (3) at (2,0)  {};
		\node[draw,fill=cyan] (4) at (1,-1.64)  {};
		\node[draw,fill=cyan] (5) at (-1,-1.64)  {};
		\node[draw,fill=cyan] (6) at (-2,0)  {};
		\node[draw] (0) at (0,0)  {};
		\node (a) at (0,-4)  {$K_{1,6}$};
		
		\draw [thick, black!70] (0)--(1);
		\draw [thick, black!70] (0)--(2);
		\draw [thick, black!70] (0)--(3);
		\draw [thick, black!70] (0)--(4);
		\draw [thick, black!70] (0)--(5);
		\draw [thick, black!70] (0)--(6);
		
		\node[draw,fill=cyan] (7) at (10.64,1) {};
		\node[draw,fill=cyan] (8) at (8.64,1)   {};
		\node[draw,fill=cyan] (9) at (8.64,-1)  {};
		\node[draw,fill=cyan] (10) at (10.64,-1)  {};
		\node[draw,fill=cyan] (11) at (12.28,2.64) {};
		\node[draw,fill=cyan] (12) at (7,2.64)   {};
		\node[draw,fill=cyan] (13) at (7,-2.64)  {};
		\node[draw,fill=cyan] (14) at (12.28,-2.64)  {};
		\node (b) at (9.64,-4)  {$K_4\circ K_1$};
		
		\draw [thick, black!70] (7)--(8)--(9)--(10)--(7);
		\draw [thick, black!70] (7)--(9);
		\draw [thick, black!70] (8)--(10);
		\draw [thick, black!70] (7)--(11);
		\draw [thick, black!70] (8)--(12);
		\draw [thick, black!70] (9)--(13);
		\draw [thick, black!70] (10)--(14);

		\node[draw,fill=cyan] (7) at (20.92,1) {};
		\node[draw,fill=cyan] (8) at (18.92,1)   {};
		\node[draw,fill=cyan] (9) at (18.92,-1)  {};
		\node[draw,fill=cyan] (10) at (20.92,-1)  {};
		\node[draw,fill=cyan] (11) at (22.56,2.64) {};
		\node[draw,fill=cyan] (12) at (17.28,2.64)   {};
		\node[draw,fill=cyan] (13) at (17.28,-2.64)  {};
		\node[draw,fill=cyan] (14) at (22.56,-2.64)  {};
		\node (c) at (19.92,-4)  {$(K_4\circ K_1)^c$};
		
		\draw [thick, black!70] (7)--(8)--(9)--(10)--(7);
		\draw [thick, black!70] (7)--(9);
		\draw [thick, black!70] (8)--(10);
		\draw [thick, black!70] (7)--(11);
		\draw [thick, black!70] (8)--(11);
		\draw [thick, black!70] (10)--(11);
		\draw [thick, black!70] (8)--(12);
		\draw [thick, black!70] (7)--(12);
		\draw [thick, black!70] (9)--(12);
		\draw [thick, black!70] (9)--(13);
		\draw [thick, black!70] (8)--(13);
		\draw [thick, black!70] (10)--(13);
		\draw [thick, black!70] (10)--(14);
		\draw [thick, black!70] (7)--(14);
		\draw [thick, black!70] (9)--(14);
		
		\node (d) at (32.71,-4)  {$X$};
\node[draw,fill=cyan] (1) at (34.71,0.83) {};
\node[draw,fill=cyan] (2) at (33.54,2)   {};
\node[draw,fill=cyan] (3) at (31.88,2)  {};
\node[draw,fill=cyan] (4) at (30.71,0.83)  {};
\node[draw,fill=cyan] (5) at (30.71,-0.83)  {};
\node[draw,fill=cyan] (6) at (31.88,-2)  {};
\node[draw,fill=cyan] (7) at (33.54,-2)  {};
\node[draw,fill=cyan] (8) at (34.71,-0.83)  {};
		\draw [thick, black!70] (1)--(2)--(3)--(4)--(5)--(6)--(7)--(8)--(1);
		\draw [thick, black!70] (1)--(3)--(5)--(7)--(1);
		\draw [thick, black!70] (2)--(4)--(6)--(8)--(2);
		\draw [thick, black!70] (1)--(4)--(7)--(2)--(5)--(8)--(3)--(6)--(1);
		\draw [thick, black!70] (1)--(5)--(1);
		\draw [thick, black!70] (2)--(6)--(2);
		\draw [thick, black!70] (3)--(7)--(3);
		\draw [thick, black!70] (4)--(8)--(4);
%%circle of radius 4
%\node[draw,fill=cyan] (9) at (38.88,1.17) {};
%\node[draw,fill=cyan] (10) at (37.99,3.33)   {};
%\node[draw,fill=cyan] (11) at (28.01,3.33)  {};
%\node[draw,fill=cyan] (12) at (27.12,1.17)  {};
%\node[draw,fill=cyan] (13) at (27.12,-1.17)  {};
%\node[draw,fill=cyan] (14) at (28.01,-3.33)  {};
%\node[draw,fill=cyan] (15) at (37.99,-3.33)  {};
%\node[draw,fill=cyan] (16) at (38.88,-1.17)  {};
%circle of radius 3.5
\node[draw,fill=cyan] (9) at (37.86,1.02) {};
\node[draw,fill=cyan] (10) at (37.08,2.91)   {};
\node[draw,fill=cyan] (11) at (28.34,2.91)  {};
\node[draw,fill=cyan] (12) at (27.56,1.02)  {};
\node[draw,fill=cyan] (13) at (27.56,-1.02)  {};
\node[draw,fill=cyan] (14) at (28.34,-2.91)  {};
\node[draw,fill=cyan] (15) at (37.08,-2.91)  {};
\node[draw,fill=cyan] (16) at (37.86,-1.02)  {};
\draw [thick, black!70] (1)--(10);
\draw [thick, black!70] (1)--(9);
\draw [thick, black!70] (1)--(15);
\draw [thick, black!70] (2)--(10);
\draw [thick, black!70] (2)--(16);
\draw [thick, black!70] (2)--(9);
\draw [thick, black!70] (3)--(11);
\draw [thick, black!70] (3)--(12);
\draw [thick, black!70] (3)--(13);
\draw [thick, black!70] (4)--(11);
\draw [thick, black!70] (4)--(12);
\draw [thick, black!70] (4)--(14);
\draw [thick, black!70] (5)--(11);
\draw [thick, black!70] (5)--(13);
\draw [thick, black!70] (5)--(14);
\draw [thick, black!70] (6)--(12);
\draw [thick, black!70] (6)--(13);
\draw [thick, black!70] (6)--(14);
\draw [thick, black!70] (7)--(9);
\draw [thick, black!70] (7)--(16);
\draw [thick, black!70] (7)--(15);
\draw [thick, black!70] (8)--(10);
\draw [thick, black!70] (8)--(15);
\draw [thick, black!70] (8)--(16);
		\end{tikzpicture}
		\vspace{-20pt}
	\end{center}	
	\caption{\label{fig2} Bidegreed split graphs with sedentary vertices marked blue}
\end{figure}

\subsection{Spectral decomposition}

The spectra of $(\alpha,\beta)$-bidegreed split graphs have been computed by Goldberg et al \cite[Propositions 4.2-4.4]{goldberg2020onsplit}. We follow this result to obtain the spectral decomposition of $(\alpha,\beta)$-bidegreed split graphs.

\begin{theorem}\label{thmspecdecombsg}
	Let $X$ be an $(\alpha,\beta)$-bidegreed split graph, where $C$ is the induced clique and $S$ is the induced independent set. Index the adjacency matrix so that $A(X)=\begin{bmatrix}
	\mathbf{J}_m-I_m & B \\ B^\top & O_n
	\end{bmatrix}$ where $m=|C|$, $n=|S|$,  $B\mathbf{1}_n=(\alpha-m+1)\mathbf{1}_m=k\mathbf{1}_m$, and $B^\top\mathbf{1}_m=\beta\mathbf{1_n}$. Then the spectral decomposition of $A(X)$ is given by
	\begin{equation*}
	A(X) =\lambda_1E_{\lambda_1}+\lambda_2E_{\lambda_2} + \displaystyle\sum_{\stackrel{\gamma\in \nspec(BB^\top)}{\lambda^2+\lambda-\gamma=0}} \lambda E_{\lambda,\gamma}
	\end{equation*}
	where $\lambda_i$, $i\in\{1,2\}$, satisfies $\lambda_i^2-(m-1)\lambda_i-\beta k=0$ and \begin{equation*}
	E_{\lambda_i}=\dfrac{1}{m\lambda_i^2+n\beta^2}\begin{bmatrix}
	\lambda_i^2 \mathbf{J} & \lambda_i \beta \mathbf{J} \\
	\lambda_i \beta \mathbf{J} & \beta^2\mathbf{J}
	\end{bmatrix},
	\end{equation*} and for $\gamma\neq0$,
	\begin{equation*}
	E_{\lambda,\gamma}=
		\displaystyle\sum_{\substack{\mathbf{1}^\top \mathbf{x}=\mathbf{0}, BB^\top \mathbf{x}=\gamma \mathbf{x}}}\dfrac{1}{\lambda^2+\gamma}\begin{bmatrix}
		\lambda^2 \mathbf{x}\mathbf{x}^\top &\lambda \mathbf{x}\mathbf{x}^\top B\\
		\lambda B^\top \mathbf{x}\mathbf{x}^\top & B^\top \mathbf{x}\mathbf{x}^\top B
		\end{bmatrix}=\sum_{\mathbf{1}^\top \mathbf{y}=\mathbf{0}, B^\top B\mathbf{y}=\gamma \mathbf{y}}\frac{\gamma}{\lambda^2+\gamma}\begin{bmatrix}
		\frac{\lambda^2}{\gamma^2}B\mathbf{y}\mathbf{y}^\top B & \frac{\lambda}{\gamma}B\mathbf{y}\mathbf{y}^\top \\ \frac{\lambda}{\gamma}\mathbf{y}\mathbf{y}^\top B & \mathbf{y}\mathbf{y}^\top
		\end{bmatrix}.
	\end{equation*}
	Moreover if $\gamma=0$ and $\lambda\in\{-1,0\}$, then
\begin{equation*}
E_{-1,0}=\displaystyle\sum_{\mathbf{1}^\top \mathbf{x}=BB^\top \mathbf{x}=\mathbf{0}}\begin{bmatrix}
\mathbf{x}\mathbf{x}^\top & O \\ O & O\end{bmatrix}\qquad\mbox{and}\qquad
E_{0,0}=\displaystyle\sum_{\mathbf{1}^\top \mathbf{y}=B^\top B\mathbf{y}=\mathbf{0}}\begin{bmatrix}
O & O \\ O & \mathbf{y}\mathbf{y}^\top\end{bmatrix}.
\end{equation*} 
\end{theorem}

\begin{proof}
By Proposition 4.2 in \cite{goldberg2020onsplit}, the eigenvalues $\lambda$ of $A(X)$ satisfy one of the following: (i) $\lambda=0$, (ii) $\lambda^2-(m-1)\lambda-\beta k=0$, or (iii) $\lambda^2+\lambda-\gamma=0$ for some $\gamma\in\nspec(BB^\top)$. From the eigenequation $\begin{bmatrix}
\mathbf{J}_m-I_m & B \\ B^\top & {O}_n
\end{bmatrix}\begin{bmatrix}
\mathbf{x} \\ \mathbf{y}
\end{bmatrix}=\lambda\begin{bmatrix}
\mathbf{x} \\ \mathbf{y}
\end{bmatrix}$, we get the following system of equations:
\begin{equation}\label{eqeigen1}
(\lambda+1)\mathbf{x}-\mathbf{J}_m\mathbf{x}-B\mathbf{y}=\mathbf{0}\quad\text{and}\quad  
-B^\top \mathbf{x} +\lambda \mathbf{y}=\mathbf{0},
\end{equation}
which together simplify to \begin{equation}\label{eqeigen2}
(\lambda^2+\lambda)\mathbf{x}-\lambda \mathbf{J}_m\mathbf{x}-BB^\top \mathbf{x}=\mathbf{0}
\end{equation} when $\lambda\neq0$. If $\mathbf{x}=\mathbf{0}$, the above system implies $\lambda=0$. So, assume first that $\mathbf{x}\neq \mathbf{0}$ and $\mathbf{1}^\top \mathbf{x}\neq \mathbf{0}$. Pre-multiplying by $\mathbf{J}_m$, and then using the fact that $BB^\top\mathbf{1}_m=\beta k\mathbf{1}_m$ gives us
\begin{equation*}\label{eqeigen3}
(\lambda^2-(m-1)\lambda-\beta k)\mathbf{J}_m\mathbf{x}=\mathbf{0}.
\end{equation*} 
Thus, an eigenvector $\bx$ satisfying $\mathbf{1}^\top \mathbf{x}\neq \mathbf{0}$ corresponds to the eigenvalue satisfying $\lambda^2-(m-1)\lambda-\beta k=0$. Observe also from \eqref{eqeigen2} that $\mathbf{x}$ must be an eigenvector of $\lambda \mathbf{J}_m+BB^\top$ and hence is a scalar multiple of $\mathbf{1}_m$ because $BB^\top$ is $\beta k$-stochastic and $\mathbf{1}^\top \mathbf{x}\neq \mathbf{0}$. (This is clear when $BB^\top$ is irreducible but can also be shown for the reducible case with a little more work.) It follows that an eigenvector of $A(X)$ corresponding to such an eigenvalue $\lambda$ is (a scalar multiple of) $\begin{bmatrix}
\lambda \mathbf{1} \\ \beta \mathbf{1}
\end{bmatrix}$ and the spectral idempotent is given by $\dfrac{1}{m\lambda^2+n\beta^2}\begin{bmatrix}
	\lambda^2 \mathbf{J} & \lambda \beta \mathbf{J} \\
	\lambda \beta \mathbf{J} & \beta^2\mathbf{J}
\end{bmatrix}$. Now, assuming that $\mathbf{x}\neq \mathbf{0}$ but $\mathbf{1}^\top \mathbf{x}=\mathbf{0}$, \eqref{eqeigen2} simplifies to $BB^\top \mathbf{x} = (\lambda^2+\lambda)\mathbf{x}$. Then, $\mathbf{x}$ must be an eigenvector of $BB^\top$ corresponding to $\gamma\in\nspec(BB^\top)$ and $\begin{bmatrix}
\lambda \mathbf{x} \\ B^\top \mathbf{x}
\end{bmatrix}$ is an eigenvector of $A(X)$ corresponding to $\lambda$ satisfying $\lambda^2+\lambda-\gamma=0$. Furthermore, the spectral idempotent for $\gamma\neq0$ is $E_{\lambda,\gamma}=
\displaystyle\sum_{\substack{\mathbf{1}^\top \mathbf{x}=\mathbf{0}, BB^\top \mathbf{x}=\gamma \mathbf{x}}}\dfrac{1}{\lambda^2+\gamma}\begin{bmatrix}
\lambda^2 \mathbf{x}\mathbf{x}^\top &\lambda \mathbf{x}\mathbf{x}^\top B\\
\lambda B^\top \mathbf{x}\mathbf{x}^\top & B^\top \mathbf{x}\mathbf{x}^\top B
\end{bmatrix}$. Taking $\mathbf{y}=B^\top \mathbf{x}$, the eigenvectors are of the form $\begin{bmatrix}
\frac{\lambda}{\gamma}B\mathbf{y} \\ \lambda \mathbf{y}
\end{bmatrix}$, and $E_{\lambda,\gamma}$ can be expressed equivalently as $\displaystyle\sum_{\mathbf{1}^\top \mathbf{y}=\mathbf{0}, B^\top B\mathbf{y}=\gamma \mathbf{y}}\frac{\gamma}{\lambda^2+\gamma}\begin{bmatrix}
\frac{\lambda^2}{\gamma^2}B\mathbf{y}\mathbf{y}^\top B & \frac{\lambda}{\gamma}B\mathbf{y}\mathbf{y}^\top \\ \frac{\lambda}{\gamma}\mathbf{y}\mathbf{y}^\top B & \mathbf{y}\mathbf{y}^\top
\end{bmatrix}$.

If $\gamma=0$, then $\lambda=-1$ or $\lambda=0$. For the case $\lambda=-1$, the eigenequation \eqref{eqeigen1} results in $B\mathbf{y} =\mathbf{0}$ and\linebreak $ B^\top \mathbf{x} + \mathbf{y}=\mathbf{0}$,
which implies that $BB^\top \mathbf{x} =\mathbf{0}$ and so $\mathbf{y}=-B^\top \mathbf{x}=\mathbf{0}$. Thus, $E_{-1,0}=\displaystyle\sum_{\mathbf{1}^\top \mathbf{x}=BB^\top \mathbf{x}=\mathbf{0}}\begin{bmatrix}
\mathbf{x}\mathbf{x}^\top & O \\ O & O\end{bmatrix}$. Similarly, when $\lambda=0$, the eigenequation \eqref{eqeigen1} simplifies to $B\mathbf{y} =\mathbf{x}$ and $ B^\top \mathbf{x}=\mathbf{0}$,
so $B^\top B \mathbf{y}=\mathbf{0}$ and $B\mathbf{y}=\mathbf{0}$. Thus,
$E_{0,0}=\displaystyle\sum_{\mathbf{1}^\top \mathbf{y}=B^\top B\mathbf{y}=\mathbf{0}}\begin{bmatrix}
O & O \\ O & \mathbf{y}\mathbf{y}^\top\end{bmatrix}.$
\end{proof}
\begin{rem}\label{rembsgsedentary}
	The complement of an $(\alpha,\beta)$-bidegreed split graph $X$ is an $(m+n-1-\beta,m+n-1-\alpha)$-bidegreed split graph whose adjacency matrix is $A(X^c)=\begin{bmatrix}
	O & \mathbf{1}_m\mathbf{1}_n^\top-B \\ \mathbf{1}_n\mathbf{1}_m^\top-B^\top & \mathbf{J}_n-I_n
	\end{bmatrix}$. Theorem~\ref{thmspecdecombsg} can then be used to find its spectral decomposition. Moreover, since $BB^\top$ is $\beta k$-stochastic, we can deduce the eigenvalues of  $(\mathbf{1}_m\mathbf{1}_n^\top-B)(\mathbf{1}_m\mathbf{1}_n^\top-B)^\top$ from that of $BB^\top$. Observe that $E_{\beta k}=\frac{1}{m}\mathbf{J}_m$ and a simple calculation yields $(\mathbf{1}_m\mathbf{1}_n^\top-B)(\mathbf{1}_m\mathbf{1}_n^\top-B)^\top
	=n\mathbf{1}_m\mathbf{1}_m^\top-\mathbf{1}_m\mathbf{1}_n^\top B^\top -B\mathbf{1}_n\mathbf{1}_m^\top +BB^\top
	=(n-2k)\mathbf{J}_m+BB^\top.$
	If $BB^\top$ is irreducible, then  $\spec((n-2k)\mathbf{J}_m+BB^\top)=\{m(n-2k)+\beta k\}\cup\spec(BB^\top)\backslash\{\beta k\}$. Otherwise, $\spec((n-2k)\mathbf{J}_m+BB^\top)=\{m(n-2k)+\beta k\}\cup\spec(BB^\top)$.
	Thus $\nspec((\mathbf{1}_m\mathbf{1}_n^\top-B)(\mathbf{1}_m\mathbf{1}_n^\top-B)^\top)=\nspec(BB^\top)$. Similarly, $\nspec((\mathbf{1}_m\mathbf{1}_n^\top-B)^\top(\mathbf{1}_m\mathbf{1}_n^\top-B))=\nspec(B^\top B).$
\end{rem}

%With the spectral decomposition of bidegreed split graphs and their complements now understood, 
We now examine $f_{u,v}(t)$ for the family of bidegreed split graphs.

\begin{ex}
    Let $X$ be a connected $(\alpha,\beta)$-bidegreed split graph where $|C|=n_1$ and $|S|=n_2$. Furthermore, suppose that $X$ is not the join of two regular graphs. The main eigenvalues of $A(X)$ are $\lambda_1,\lambda_2$ satisfying $\lambda_i^2-(n_1-1)\lambda_i-\beta k=0$. Note that $\lambda_1\neq \lambda_2$, so by Proposition~\ref{prop:2mainevec} $A(X)$ has exactly two linearly independent main eigenvectors. The graph $X$ then meets the hypotheses of Corollary~\ref{cor:twomain bidegree bound}, and so
    \[f_{u,v}(t)\leq \begin{cases}
        2/n_1, &u,v\in C\\
        2/n_2, &u,v\in S\\
        2/\sqrt{n_1n_2}, &\text{otherwise}.
    \end{cases}\]
    To compute $f_{u,v}(t)$ exactly, we use the spectral decomposition of $A(X)$ and $A(X^c)$. As stated in the proof of Theorem~\ref{thmspecdecombsg}, the associated eigenvectors of $\lambda_1,\lambda_2$ are $\begin{bmatrix}
        \lambda_i\one \\
        \beta \one 
    \end{bmatrix}.$ The graph $X^c$ is disconnected only if $X$ is the join of two regular graphs, and therefore $X^c$ is a connected bidegreed split graph. Then $A(X^c)$ has two linearly independent main eigenvectors $\begin{bmatrix}
        \lambda_i^c\one \\
        \beta^c \one 
    \end{bmatrix}$ where $\beta^c = (n_1+n_2-1-\alpha)$ and $\lambda_1^c,\lambda_2^c$ satisfy $(\lambda_i^c)^2-(n_2-1)\lambda_i^c-\beta^c(n_1-\beta)=0$. By Lemma~\ref{lem: transition_difference},
    \[U_{X^c}(t) - e^{-it}U_X(-t) = e^{it\lambda_1^c}E_{\lambda_1^c}+e^{it\lambda_2^c}E_{\lambda_2^c}-e^{-it(\lambda_1+1)}E_{\lambda_1} - e^{-it(\lambda_2+1)}E_{\lambda_2}\]
    where 
    \[E_{\lambda_i^c} = \frac{1}{n_2(\lambda_i^c)^2 + n_1(\beta^c)^2}\begin{bmatrix}
        (\beta^c)^2\mathbf{J} & \lambda_i^c\beta^c\mathbf{J}\\
        \lambda_i^c\beta^c\mathbf{J} & (\lambda_i^c)^2\mathbf{J}\\
    \end{bmatrix}.\]
    We thus obtain 
    \[f_{u,v}(t) =  \begin{cases}
        \left|\frac{e^{it\lambda_1^c}}{n_1+n_2(\lambda_1^c/\beta^c)^2} 
        + \frac{e^{it\lambda_2^c}}{n_1+n_2(\lambda_2^c/\beta^c)^2}
        -\frac{e^{-it(\lambda_1+1)}}{n_1+n_2(\beta/\lambda_1)^2}
        -\frac{e^{-it(\lambda_2+1)}}{n_1+n_2(\beta/\lambda_2)^2}\right|,& u,v \in C\\
        
        \left|\frac{e^{it\lambda_1^c}}{n_2+n_1(\lambda_1^c/\beta^c)^2} 
        + \frac{e^{it\lambda_2^c}}{n_2+n_1(\lambda_2^c/\beta^c)^2}
        -\frac{e^{-it(\lambda_1+1)}}{n_2+n_1(\beta/\lambda_1)^2}
        -\frac{e^{-it(\lambda_2+1)}}{n_2+n_1(\beta/\lambda_2)^2}\right|,& u,v \in C\\
        
        \left|\frac{e^{it\lambda_1^c}}{n_2\lambda_1^c/\beta^c + n_1\beta^c/\lambda_1^c} + \frac{e^{it\lambda_2^c}}{n_2\lambda_2^c/\beta^c + n_1\beta^c/\lambda_2^c}-\frac{e^{-it(\lambda_1+1)}}{n_1\lambda_1/\beta+n_2\beta/\lambda_1}-\frac{e^{-it(\lambda_2+1)}}{n_1\lambda_2/\beta+n_2\beta/\lambda_2}\right|,&\text{otherwise}.
    \end{cases}\]  
%This explicit calculation of $f_{u,v}(t)$ is rather cumbersome. However, i
If one applies the triangle inequality and some natural upper bounds on the summands of $f_{u,v}(t)$, the upper bounds obtained for $f_{u,v}(t)$ above are twice those found using Corollary~\ref{cor:twomain bidegree bound}.
\end{ex}

\subsection{Periodicity}

Let $X$ be a bidegreed split graph. Using the form of the eigenvalues $\lambda_1,\lambda_2$ and the $\lambda$s of $X$, we deduce that the eigenvalue support of each vertex $u$ of $X$ is given by
\begin{center}
$\supp_{u}= 
\begin{cases}
    \bigg\{\frac{m-1\pm\sqrt{(m-1)^2+4\beta k}}{2}\bigg\}\bigcup\left\{\frac{-1\pm\sqrt{1+4\gamma}}{2}:\gamma\in \nspec(BB^\top)\right\}, & \text{if } \gamma\neq 0 \\
    \{-1\}\bigcup\left\{\frac{m-1\pm\sqrt{(m-1)^2+4\beta k}}{2}\right\}\bigcup\left\{\frac{-1\pm\sqrt{1+4\gamma}}{2}:\gamma\in \nspec(BB^\top)\right\},  & \text{if $\gamma=0$ and $u\in C$} \\
    \{0\}\bigcup\left\{\frac{m-1\pm\sqrt{(m-1)^2+4\beta k}}{2}\right\}\bigcup\left\{\frac{-1\pm\sqrt{1+4\gamma}}{2}:\gamma\in \nspec(BB^\top)\right\},  & \text{if $\gamma=0$ and $u\in S$}.
\end{cases}$    
\end{center}

Note that a bidegreed split graph with induced clique of size $m=1$ is a star, which is known to be periodic. If $m\geq 2$, then we have the following characterization of periodicity in bidegreed split graphs.

\begin{theorem}
\label{thm:splitper}
Suppose $m\geq 2$. A bidegreed split graph $X$ with induced clique of size $m$ is periodic if and only if $(m-1)^2+4\beta k$ and $1+4\gamma$ are perfect squares for all $\gamma\in \nspec(BB^\top)$.

\begin{proof}
Recall that vertex $u$ is periodic in $X$ if and only if $\supp_{u}$ satisfies the ratio condition. If $\gamma=0$,
then $0$ or $-1$ belongs to $\supp_u$, and $\supp_u$ satisfies the ratio condition if and only if $\supp_u\subseteq\Z$. Now, suppose $\gamma\neq 0$ and let $\Delta_1=(m-1)^2+\beta k$ and $\Delta_2=1+4\gamma$. Then $\supp_u$ satisfying the ratio condition requires that 
\begin{center}
$\displaystyle\frac{\big(\frac{m-1+\sqrt{\Delta_1}}{2}\big)-\big(\frac{-1+\sqrt{\Delta_2}}{2}\big)}{\big(\frac{m-1+\sqrt{\Delta_1}}{2}\big)-\big(\frac{m-1-\sqrt{\Delta_1}}{2}\big)}\in\Q$.
\end{center}
Equivalently, $\frac{m+\sqrt{\Delta_1}-\sqrt{\Delta_2}}{\sqrt{\Delta_1}}\in\Q$, which holds if and only if $\Delta_1,\Delta_2$ are both perfect squares.
\end{proof}
\end{theorem}

%Since periodicity is a necessary condition for PST, Theorem \ref{thm:splitper} may be used to directly rule out PST in bidegreed split graphs. 
Notice that Theorem \ref{thm:splitper} puts strict requirements on the eigenvalues of a bidegreed split graph to admit PST, making it an unlikely property for this family of graphs. This motivates to instead investigate sedentariness in bidegreed split graphs in the next subsection.

\subsection{Sedentariness}

Next, we examine sedentariness in connected bidegreed split graphs. We show that connected bidegreed split graphs with at most four eigenvalues and their graph complements admit sedentary vertices.

\begin{theorem}\label{thmbsgsedentary}
Let $X$ be a bidegreed split graph such that $\nspec(BB^\top)=\{\gamma\}$ (resp. $B^\top B$). If $m>\sqrt{1+4\gamma}+1$ (resp. $n>\sqrt{1+4\gamma}+1$), then the vertices in $C$ (resp. $S$) are sedentary in $X$ and $X^c$. 
\end{theorem}

\begin{proof}
	Let $X$ be an $(\alpha,\beta)$-bidegreed split graph where $C$ is an induced clique, $S$ is an induced independent set, and $V(C)\cup V(S)=V(X)$. Theorem~\ref{thmspecdecombsg} provides the spectral decomposition of any bidegreed split graph. Consider the spectral idempotent 
	\begin{equation*}
	E_{\lambda,\gamma}=\displaystyle\sum_{\substack{\mathbf{1}^\top \mathbf{x}=\mathbf{0}, BB^\top \mathbf{x}=\gamma \mathbf{x}}}\dfrac{1}{\lambda^2+\gamma}\begin{bmatrix}
	\lambda^2 \mathbf{x}\mathbf{x}^\top &\lambda \mathbf{x}\mathbf{x}^\top B\\
	\lambda B^\top \mathbf{x}\mathbf{x}^\top & B^\top \mathbf{x}\mathbf{x}^\top B
	\end{bmatrix}\label{eqorthprojElambdagamma}
	\end{equation*}
	where $0\neq\gamma\in \nspec(BB^\top)$ and $\lambda$ is a root of $t^2+t-\gamma=0$. The sum above is taken over a set of orthonormal eigenvectors $\mathbf{x}$ corresponding to the eigenvalue $\gamma\in \nspec(BB^\top)$. Assume $BB^\top$ is irreducible. (Otherwise, if $BB^\top$ is reducible and $\nspec(BB^\top)=\{\gamma\}$, then $BB^\top= I_m$ and $\gamma=1$, and the same following argument applies.) Since $\nspec(BB^\top)=\{\gamma\}$, $ \sum_{\mathbf{x}} \mathbf{x}\mathbf{x}^\top=I-\frac{1}{m}J_m$.
	
We first show that all vertices in $C$ are sedentary. Set $\lambda=\frac{-1-\sqrt{1+4\gamma}}{2}$. Then $\frac{\lambda^2}{\lambda^2+\gamma}
 =\frac{1}{2}+\frac{1}{2\sqrt{1+4\gamma}}>\frac{1}{2}.$ So the first diagonal block of $E_{\lambda,\gamma}$ satisfies $\frac{\lambda^2}{\lambda^2+\gamma}\sum_{\mathbf{x}}  \mathbf{x}\mathbf{x}^\top=\big(\frac{1}{2}+\frac{1}{2\sqrt{1+4{\gamma}}}\big)\left(I-\frac{1}{m}J_m\right)$. The diagonal entries of the first diagonal block of $E_{\lambda,\gamma}$ each have a value at least $\big(\frac{1}{2}+\frac{1}{2\sqrt{1+4{\gamma}}}\big)\big(1-\frac{1}{m}\big)$, This is greater than $\frac{1}{2}$ when $m>\sqrt{1+4{\gamma}}+1$. The sedentariness conclusion for vertices in $C$ follows from Theorem~\ref{thm:sed}. The same argument also works when $\gamma=0$. Furthermore, by Remark~\ref{rembsgsedentary} since $\nspec(BB^\top)=\nspec((\mathbf{1}_m\mathbf{1}_n^\top-B)(\mathbf{1}_m\mathbf{1}_n^\top-B)^\top$, it follows that the vertices in $C$ are also sedentary as the vertices in $X^c$. Note that vertices in $C$ form an induced independent set in $X^c$ instead of a clique.
	
	To show that the vertices in $S$ are sedentary when $\nspec(B^\top B)=\{\gamma\}$, let $\gamma\neq 0$ and consider the matrix $E_{\lambda,\gamma}=\displaystyle\sum_{\mathbf{1}^\top \mathbf{y}=0, B^\top B\mathbf{y}=\gamma \mathbf{y}}\frac{\gamma}{\lambda^2+\gamma}\begin{bmatrix}
	\frac{\lambda^2}{\gamma^2}B\mathbf{y}\mathbf{y}^\top B & \frac{\lambda}{\gamma}B\mathbf{y}\mathbf{y}^\top \\ \frac{\lambda}{\gamma}\mathbf{y}\mathbf{y}^\top B & \mathbf{y}\mathbf{y}^\top
	\end{bmatrix}$.
Take $\lambda=\frac{-1+\sqrt{1+4\gamma}}{2}$. Similarly to the previous case, we have $\frac{\lambda^2}{\lambda^2+\gamma}=\frac{1}{2}-\frac{1}{2\sqrt{1+4\gamma}}$. So, $\frac{\gamma}{\lambda^2+\gamma}=\frac{1}{2}+\frac{1}{2\sqrt{1+4\gamma}}>\frac{1}{2}.$ Since $\sum_{\mathbf{y}}\mathbf{y}\mathbf{y}^\top=I-\frac{1}{n}\mathbf{J}_n$, the diagonal entries of the last diagonal block of $E_{\lambda,\gamma}$ satisfy:
	$\frac{\gamma}{\lambda^2+\gamma}\sum_{\mathbf{y}}  \mathbf{y}\mathbf{y}^\top =(\frac{1}{2}+\frac{1}{2\sqrt{1+4{\gamma}}})(I-\frac{1}{n}\mathbf{J}_n).$
Sedentariness then follows from Theorem~\ref{thm:sed} provided $n>\sqrt{1+4\gamma}+1$. If $\gamma=0$, it is evident from $E_0=\displaystyle\sum_{\mathbf{1}^\top \mathbf{y}=B^\top B\mathbf{y}=\mathbf{0}}\begin{bmatrix}
	O & O \\ O & \mathbf{y}\mathbf{y}^\top
	\end{bmatrix}$
that the diagonal entries of the nonzero block are $\left(1-\frac{1}{n}\right)$, which is greater than $\frac{1}{2}$ when $n>2$. By Remark~\ref{rembsgsedentary}, the vertices in $S$ are also sedentary in $X^c$.
\end{proof}

%\begin{ex}
%$X=K_m\vee O_n$ is a bidegreed split graph with $B=\mathbf{1}_m\mathbf{1}_n^\top$, $|C|=m$, and $|S|=n$. We have $\nspec(BB^\top)=\nspec(B^\top B)=\{0\}$ so by Theorem~\ref{thmbsgsedentary}, vertices in $C$ (resp. $S$) are sedentary in $X$ (and $X^c$ but $X^c$ is disconnected) when $m> 2 = \sqrt{1+4(0)}+1$ (resp. $n>2$).
%\end{ex}

In \cite{goldberg2020onsplit}, connected bidegreed split graphs of diameter $3$ with four distinct eigenvalues were classified . Such graphs satisfy $|\nspec(BB^\top)|=1$ and so Theorem~\ref{thmbsgsedentary} applies as seen in the next example. %In what follows, $X\circ Y$ denotes the corona product \cite{frucht1970corona}.

%\begin{ex}
%The graph $X=K_n\circ K_1$ is a bidegreed split graph with $B=I_n$ and $|C|=|S|=n$. Hence $BB^\top=B^\top B=I_n$ and $\nspec(BB^\top)=\nspec(B^\top B)=\{1\}$. By Theorem~\ref{thmbsgsedentary}, all vertices in $X$ and $X^c$ are sedentary when $n\geq4>\sqrt{5}+1$. See Fig.~\ref{fig2}.
%\end{ex}

\begin{ex}
The graph $X$ with adjacency matrix $\begin{bmatrix} \mathbf{J}_m-I_m & B_D\\ B_D^\top & O\end{bmatrix}$ where $B_D$ is the incidence matrix of a $(v=m,b=n,r,k=\beta,\lambda)$-design $D$, is a bidegreed split graph with $\gamma=\frac{(m-\beta)(\alpha-m+1)}{m-1}$ By Theorem~\ref{thmbsgsedentary}, vertices in $C$ (resp. $S$) are sedentary in $X$ and $X^c$ if $m>\sqrt{1+4\gamma}+1$ (resp. $n>\sqrt{1+4\gamma}+1$).
\end{ex}

We now construct bidegreed split graphs with at least five distinct eigenvalues that admit sedentariness.

\begin{cor}\label{corbsgsedentary}
	Let ${X}$ be a bidegreed split graph such that $BB^\top$ (resp. $B^\top B$) is reducible into $k$ subblocks and each irreducible subblock $P_i$ of size $n_i$ satisfies $\nspec(P_i)=\{\gamma_i\}$. Then the vertices in the clique $C$ (resp. $S$) corresponding to the subblock $P_i$ are sedentary in ${X}$ and in ${X}^c$ whenever $n_i>\sqrt{1+\gamma_i}+1$.
\end{cor}

\begin{proof}
	We extend the technique used in the proof of Theorem~\ref{thmbsgsedentary}. We have the following key observations: (i) each irreducible subblock $P_i$ of $BB^\top$ is $\beta k$-stochastic, (ii) $\nspec(BB^\top)=\{\beta k\}\cup\left(\bigcup_i \nspec{(P_i)}\right)$, and (iii) if $\mathbf{x}_i$ is a non-main eigenvector of $P_i$ corresponding to $\gamma_i$, then $\mathbf{x}=[\mathbf{0}^\top\cdots\ \mathbf{0}^\top\mathbf{x}_i^\top\ \mathbf{0}^\top\cdots\mathbf{0}^\top]^\top$, partitioned conformally with the subblocks of $BB^\top$, is a non-main eigenvector for $BB^\top$ corresponding to $\gamma_i$. Now, the sum $\sum_{\mathbf{x}}\mathbf{x}\mathbf{x}^\top$ gives us a block diagonal matrix with all zero blocks except at the $i$th diagonal block which is $I_{n_i}-\frac{1}{n_i}\mathbf{J}_{n_i}$ where $n_i$ is the size of the subblock $P_i$. Via the same argument as in the proof of Theorem~\ref{thmbsgsedentary}, we get $\frac{\lambda^2}{\lambda^2+\gamma_i}\sum_{\mathbf{x}}\mathbf{x}\mathbf{x}^\top=(\frac{1}{2}+\frac{1}{2\sqrt{1+4\gamma_i}})(I_{n_i}-\frac{1}{n_i}\mathbf{J}_{n_i})$ so the diagonal entries corresponding to the vertices in $C$ and the subblock $P_i$ in $E_{\lambda,\gamma_i}$ have value $(\frac{1}{2}+\frac{1}{2\sqrt{1+4\gamma_i}})(1-\frac{1}{n_i})>\frac{1}{2}$ when $n_i>\sqrt{1+4\gamma_i}+1$. The result using $B^\top B$ follows similarly.
\end{proof}

\begin{ex}
	Let $X$ be the bidegreed split graph where $B=(\mathbf{J}_{n_1}-I_{n_1})\oplus(\mathbf{J}_{n_1}-I_{n_1})$. We have $BB^\top=B^\top B=((n_1-2)\mathbf{J_{n_1}}+I_{n_1})\oplus((n_1-2)\mathbf{J_{n_1}}+I_{n_1})$ which is reducible. Note that $\nspec(BB^\top)=\{n_1(n_1-2)+1,1\}$ so Theorem~\ref{thmbsgsedentary} does not directly apply. However as $\nspec(((n_1-2)\mathbf{J_{n_1}}+I_{n_1}))=\{1\}$ by Corollary~\ref{corbsgsedentary}, the vertices in $C$ and $S$ corresponding to the block $((n_1-2)\mathbf{J_{n_1}}+I_{n_1})$ are sedentary in $X$ and $X^c$ when $n_i\geq 4>\sqrt{5}+1$. The bidegreed split graph $X$ when $n_1=4$ is illustrated in Fig.~\ref{fig2}.
\end{ex}

\section{Open questions}
\label{sec:op}

In this paper, we investigated the behavior of quantum walks under joins, Cartesian products, and graph complements. We characterized conditions such that PGST and peak state transfer are preserved under these graph operations. We also examined the difference between the behavior of quantum walks on a graph and its complement by providing bounds on $f_{u,v}(t)$, $g_{u,v}(t)$, $\sigma_{u,v}(t)$ and $\rho_{u,v}(t)$. In order to inspire further work on this topic, we pose the following open questions. 

In \cite{bhattacharjya2024quantum,monterde2025laplacian}, it was shown that the blow-up operation does not preserve PGST relative to $A$ and $L$. So we broadly ask: under which conditions are PGST and peak state transfer preserved under direct, rooted, wreath, and corona products? For the case where peak state transfer is preserved under joins (resp., Cartesian products), the transfer probability does not exceed that of the underlying graph (resp., the product of the transfer probabilities of the underlying graphs). Are there examples where peak state transfer is preserved after a graph operation with increased transfer probability (that is, the $(u,v)$ entry of the bounding matrix of the resulting graph is larger than that of the underlying graph)?

To deepen the connection between quantum walks and main eigenvalues, we would like to characterize the graphs that attain the upper bound for $f_{u,v}(t)$ in Theorem~\ref{thm:projector upper bound}. A related open question is, which graphs on $n$ vertices have $k\geq 3$ linearly independent main eigenvectors?

Finally, since paths are riveting to researchers, it is desirable to derive tight upper bounds for $f_{u,v}(t)$, $g_{u,v}(t)$, $\sigma_{u,v}(t)$, and $\rho_{u,v}(t)$ for $P_n$. For $\sigma_{u,v}(t)$ and $\rho_{u,v}(t)$, the interesting case is when PGST does not occur between $u$ and $v$ in $P_n$ and its complement relative to the same sequence of times, because in this case $\sigma_{u,v}(t)$ and $\rho_{u,v}(t)$ are bounded away from 2 and 1 respectively, by Proposition~\ref{prop:NG}(1A). 

\bigskip

\noindent \textbf{Acknowledgments.}
We are grateful to the Pacific Institute for the Mathematical Sciences for the First Year Interest Group (FYIG) Initiative on the topic of \textit{Quantum Walks on Graphs} led by H.~Monterde, where we all started to work on this project. J.~Cooper is supported by the University of Victoria Department of Mathematics and Statistics and Faculty of Graduate Studies. H.~De Vera is supported by Dr.~Steve Kirkland, the Department of Mathematics, and the Faculty of Graduate and Postdoctoral Studies at the University of Manitoba. H.~Monterde is supported by the Pacific Institute for the Mathematical Sciences through the PIMS-Simons Postdoctoral Fellowship. S.A.~Talebpour is supported by Dr.~Karen Meagher, and the Faculty of Graduate Studies and Research at the University of Regina. X.~Wang was supported by the University of British Columbia Department of Mathematics for the whole duration of this project.

\bibliographystyle{alpha}
\bibliography{mybibfile}
\end{document}